\documentclass[11pt]{article}

\usepackage{authblk}
\usepackage{fullpage}  
\usepackage{float}
\usepackage{pb-diagram}  		
\usepackage{esvect}
\usepackage[utf8]{inputenc} 
\usepackage[T1]{fontenc}    
\usepackage{hyperref}       
\usepackage{url}            
\usepackage{booktabs}       
\usepackage{amsfonts}       
\usepackage{nicefrac}       
\usepackage{microtype}      
\usepackage{fullpage}  
\usepackage{float}
\usepackage{pb-diagram}  
\usepackage{makecell}		

\usepackage{url}

\usepackage{algorithm, algorithmicx,algpseudocode, listings} 
\usepackage{multirow} 
\usepackage{hhline}  
\usepackage{amsmath}
\usepackage{xcolor}

\usepackage{graphicx}
\usepackage{amscd}
\usepackage{amssymb,mathrsfs}
\input{epsf.sty}
\usepackage{amsthm,amscd}
\usepackage{color}
\usepackage{latexsym}
\usepackage{epic}
\usepackage{appendix}
\usepackage{enumerate}
\usepackage{longtable}
\usepackage{lscape}
\usepackage{extarrows}
\usepackage{epstopdf}
\usepackage{listings}
\newlength\myindent

\renewcommand{\theequation}{\arabic{equation}}

\newcommand{\Rmnum}[1]{\expandafter\@slowromancap\romannumeral #1@}
\newcommand{\inner}[3][]{{\langle #2,#3 \rangle_{#1}}}

\allowdisplaybreaks
\newcommand\numberthis{\addtocounter{equation}{1}\tag{\theequation}}

\newcommand{\whcomm}[2]{{}{#2}}
\newcommand{\whnew}[2]{{}{#2}}

\newcommand{\kwrev}[1]{{#1}}

\newcommand{\slversions}[2]{{}{#2}} 

\newtheorem{definition}{Definition}[section]
\newtheorem{theorem}{Theorem}[section]
\newtheorem{lemma}{Lemma}[section]

\newtheorem{assumption}{Assumption}[section]

\numberwithin{equation}{section}

\DeclareMathOperator{\E}{\mathrm{F}}
\DeclareMathOperator{\F}{\mathrm{F}}
\DeclareMathOperator{\T}{\mathrm{T}}
\DeclareMathOperator{\Hess}{\mathrm{Hess}}
\DeclareMathOperator{\grad}{\mathrm{grad}}

\DeclareMathOperator{\Exp}{\mathrm{Exp}}
\DeclareMathOperator{\Prox}{\mathrm{Prox}}

\DeclareMathOperator{\N}{\mathrm{N}}
\DeclareMathOperator{\D}{\mathrm{D}}
\DeclareMathOperator{\dist}{\mathrm{dist}}
\DeclareMathOperator{\trace}{\mathrm{trace}}

\DeclareMathOperator{\St}{\mathrm{St}}
\DeclareMathOperator{\Gr}{\mathrm{Gr}}

\DeclareMathOperator*{\argmin}{arg\,min}

\begin{document}

\title{An Inexact Riemannian Proximal Gradient Method
\footnotetext{Corresponding authors: Wen Huang (\texttt{wen.huang@xmu.edu.cn})  and Ke Wei (\texttt{kewei@fudan.edu.cn}). WH was partially  supported by the Fundamental Research Funds for the Central Universities (NO. 20720190060) and National Natural Science Foundation of China (NO. 12001455). KW was partially  supported by the NSFC Grant 11801088 and the Shanghai Sailing Program 18YF1401600.}}
\author[1]{Wen Huang}
\author[2]{Ke Wei}

\affil[1]{ School of Mathematical Sciences, Xiamen University, Xiamen, China.\vspace{.15cm}}
\affil[2]{School of Data Science, Fudan University, Shanghai, China.}

\maketitle


\begin{abstract}
This paper considers the problem of minimizing the summation of a differentiable function and a nonsmooth function on a Riemannian manifold. In recent years, proximal gradient method and its invariants have been generalized to the Riemannian setting for solving such problems. \kwrev{Different approaches to generalize the proximal mapping to the Riemannian setting
lead versions of Riemannian proximal gradient methods.} However, their convergence analyses all rely on  solving their Riemannian proximal mapping exactly, which is either too expensive or impracticable. In this paper, we \kwrev{study the convergence of} an inexact Riemannian proximal gradient method. It is proven that if the proximal mapping is solved sufficiently accurately, then the global convergence and local convergence rate based on the Riemannian Kurdyka-\L ojasiewicz property can be guaranteed. \kwrev{Moreover, practical conditions on the accuracy for solving the Riemannian proximal mapping are provided. As a byproduct, the proximal gradient method on the Stiefel manifold proposed in \cite{CMSZ2019} can be viewed as the inexact Riemannian proximal gradient method provided the proximal mapping is solved to certain accuracy. }
Finally, numerical experiments on sparse principal component analysis are conducted to test the proposed practical conditions. 
\end{abstract}

\section{Introduction}

Proximal gradient method and its variants are family of  efficient algorithms for composite optimization problems of the form 
\begin{equation}\label{prob0}
\min_{x \in \mathbb{R}^n} F(x) = f(x) + g(x),
\end{equation}
where $f$ is  differentiable, and $g$ is continuous but could be nonsmooth.
In the simplest form, the method updates the iterate via 
\begin{align}
&\left\{
\begin{array}{ll}
	d_k = \argmin_{p \in \mathbb{R}^{n}} \inner[\F]{\nabla f(x_k)}{p} + \frac{L}{2} \|p\|^{2}_{\F} + g(x_k + p), & \hbox{(Proximal mapping\footnotemark)}  \\
	x_{k+1} = x_k + d_k, & \hbox{(Update iterates)} \label{RPG:EPG}
\end{array}
\right.
\end{align}
where $\inner[\F]{u}{v} = u^T v$ and $\|u\|_{\F}^2 = \inner[\F]{u}{u}$.
The idea is to simplify the objective function in each iteration by replacing the differentiable term $f$ with its first order approximation around the current iterate. In many applications, the proximal mapping has a closed-form solution or can be computed efficiently. Thus, the algorithm has low per iteration cost and is applicable for large-scale problems. 
For convergence analysis of proximal gradient methods, we refer the interested readers to \cite{Beck2009,Beck2017,Darzentas1983Problem,Nesterov83,AB2009,LL2015,GL2016} and references therein.

This paper considers  a  problem similar to \eqref{prob0} but with a manifold constraint,
\begin{equation} \label{prob1}
\min_{x \in \mathcal{M}} F(x) = f(x) + g(x),
\end{equation}
where $\mathcal{M}$ is a finite dimensional Riemannian manifold. Such optimization problem is of interest due to many important applications including but not limit to compressed models~\cite{OLCO2013}, sparse principal component analysis~\cite{ZHT2006,HuaWei2019}, sparse variable principal component analysis~\cite{US2008,CMW2013,XLY2020}, discriminative $k$-means~\cite{YZW2008}, texture and imaging inpainting~\cite{LRZM2012}, co-sparse factor regression~\cite{MDC2017}, and low-rank sparse coding~\cite{ZGLXA2013,SQ2016}.

In the presence of the manifold constraints, developing Riemannian proximal gradient methods is more difficult due to nonlinearity of the domain. The update formula in~\eqref{RPG:EPG} can be generalized to the Riemannian setting using a standard technique, i.e., via the notion of retraction. 
However, generalizing the proximal mapping to the Riemannian setting is not straightforward and different versions have been proposed.
 In~\cite{CMSZ2019}, a proximal gradient method on the Stiefel manifold called \emph{ManPG}, is proposed and analyzed by generalizing the proximal mapping~\eqref{RPG:EPG} to
\begin{equation} \label{e21}
\eta_k = \argmin_{\eta \in \T_{x_k} \mathcal{M}} \inner[\F]{\nabla f(x_k)}{\eta} + \frac{\tilde{L}}{2} \|\eta\|^2_{\F} + g(x_k + \eta)
\end{equation}
via the  restriction of the search direction $\eta$ onto the tangent space at $x_k$. It is shown that such proximal mapping can be solved efficiently by a semi-smooth Newton method when the manifold $\mathcal{M}$ is the Stiefel manifold. 
In~\cite{HuaWei2019}, a diagonal weighted proximal mapping is defined by replacing $\|\eta\|_F^2$in \eqref{e21} with $\inner[\F]{\eta}{W \eta}$, where the diagonal weighted linear operator $W$ is  carefully selected. Moreover, the Nesterov momentum acceleration technique is further introduced to accelerate the algorithm, yielding an algorithm called \emph{AManPG}. Note that 
the Riemannian proximal mappings~\eqref{e21} involves the calculation of the addition, i.e., $x_k + \eta$, which cannot be defined on a generic manifold. 
In~\cite{HuaWei2019b}, a Riemannian proximal gradient method, called \emph{RPG}, is proposed by replacing the addition $x_k + p$ with a retraction $R_{x_k}(\eta)$, so that it is well-defined for generic manifolds. In addition, the Riemannian metric $\inner[x]{}{}$ is further used instead of the Euclidean inner product $\inner[\F]{}{}$, and a stationary point is used instead of a minimizer. \kwrev{ More precisely,   letting 
\begin{align*}
    \ell_{x_k}(\eta):=\inner[x_k]{\nabla f(x_k)}{\eta} + \frac{\tilde{L}}{2} \inner[x_k]{\eta}{\eta} + g(R_{x_k}(\eta)),
\end{align*}
the Riemannian proximal mapping in RPG is given by 
\begin{align}\label{e23}
\eta_k \in \T_{x_k} \mathcal{M} \hbox{ is a stationary point of $\ell_{x_k}(\eta)$ that satisfies }\ell_{x_k}(\eta_k)\leq \ell_{x_k}(0).
\end{align}}Unlike ManPG and AManPG that only guarantee global convergence, the local convergence rate of RPG has also been established in terms of  Riemannian KL property. 

 The convergence analyses of Riemannian proximal gradient methods in~\cite{CMSZ2019,HuaWei2019,HuaWei2019b} all rely on  solving proximal mappings~\eqref{e21} and~\eqref{e23} exactly. On the one hand, solving these proximal mappings exactly is not practicable due to rounding errors from the finite precision arithmetic. On the other hand, these Riemannian proximal mappings generally do not admit a closed-form solution, and finding a highly accurate solution may take too much computational time. 
\kwrev{Therefore, it makes sense to study the convergence of the inexact Riemannian proximal gradient method (i.e., the  method without solving the proximal mapping~\eqref{e23} exactly), which is essentially the goal of this paper. The main contributions of this paper can be summaries as follows:
\begin{itemize}
    \item A general framework of the inexact RPG method is presented in Section~\ref{sect:IRPG}. The global convergence as well as the local convergence rate of the method  are respectively studied in Sections~\ref{subsec:global} and \ref{subsec:local} based on different theoretical conditions. The local convergence analysis is based on the Riemannian KL property.
    \item It is shown in Section~\ref{sect:subprobGlobal} that if we solve \eqref{e21} to certain accuracy, the global convergence of the inexact RPG can be guaranteed. As a result,  ManPG  in~\cite{CMSZ2019} can be viewed as the inexact RPG method with the proximal mapping~\eqref{e23}, and it is not necessary to solve \eqref{e21} exactly for ManPG to enjoy global convergence. 
    \item Under the assumption $g$ is retraction convex, a practical condition which meets the requirement for the local convergence rate analysis is provided in Section~\ref{sec:LocalRate}. The condition is derived based on the notion of error bound. 
\end{itemize}
} 



Inexact proximal gradient methods have been investigated in the Euclidean setting, see e.g., \cite{Com2004,FP2011,SRB2011,VSBV2013,BPR2020}. Multiple practical termination criteria for the inexact proximal mapping have been given such that the global convergence and local convergence rate are preserved. However, these criteria, the corresponding theoretical results, \whcomm{}{and the algorithm design} all rely on the convexity of the function in the proximal mapping. Therefore, these methods can not be trivially generalized to the Riemannian setting since the objective function in the Riemannian proximal mapping~\eqref{e23} may not be convex due to the existence of a retraction. \whcomm{}{Note that for the inexact Riemannian proximal gradient method, the global and local convergence analyses and the 
condition that guarantees global convergence all do not assume convexity of the Riemannian proximal mapping. 
The convexity assumption is only made for the algorithm design that guarantees local convergence rate. 
}

The rest of this paper is organized as follows.  Notation and preliminaries about manifolds are given in Section~\ref{sect:NotationPreliminaries}. The inexact Riemannian proximal gradient method  is presented in Section~\ref{sect:IRPG}, followed by the convergence analysis. Section~\ref{sect:subprob} gives practical conditions on the accuracy  for solving the inexact Riemannian proximal mapping when the manifold has a linear ambient space. Numerical experiments are presented in Section~\ref{sect:NumExp} to test the practical conditions.

\section{Notation and Preliminaries on Manifolds} \label{sect:NotationPreliminaries}
The Riemannian concepts of this paper follow from the standard literature, e.g.,~\cite{Boo1986,AMS2008} and the  related notation  follows from~\cite{AMS2008}. A Riemannian manifold $\mathcal{M}$ is a manifold endowed with a Riemannian metric $(\eta_x, \xi_x) \mapsto \inner[x]{\eta_x}{ \xi_x} \in \mathbb{R}$, where $\eta_x$ and $\xi_x$ are tangent vectors in the tangent space of $\mathcal{M}$ at $x$. The induced norm in the tangent space at $x$ is denoted by $\|\cdot\|_x$ or $\|\cdot\|$ when the subscript is clear from the context.
The tangent space of the manifold $\mathcal{M}$ at $x$ is denoted by $\T_x \mathcal{M}$, and the tangent bundle, which is the set of all tangent vectors, is denoted by $\T \mathcal{M}$. A vector field is a function from the manifold to its tangent bundle, i.e., $\eta:\mathcal{M} \rightarrow \T \mathcal{M}: x\mapsto \eta_x$. An open ball on a tangent space is denoted by $\mathcal{B}(\eta_x, r) = \{\xi_x \in \T_{x} \mathcal{M} \mid \|\xi_x - \eta_x\|_x < r\}$. 
An open ball on the manifold is denoted by $\mathbb{B}(x, r) = \{ y \in \mathcal{M} \mid \dist(y, x) < r \}$, where $\dist(x, y)$ denotes the distance between $x$ and $y$ on $\mathcal{M}$. 

A retraction is a smooth ($C^\infty$) mapping from the tangent bundle to the manifold such that (i)~$R(0_x) = x$ for all $x \in \mathcal{M}$, where $0_x$ denotes the origin of $\T_x \mathcal{M}$, and (ii) $\frac{d}{d t} R(t \eta_x) \vert_{t = 0} = \eta_x$ for all $\eta_x \in \T_x \mathcal{M}$. The domain of $R$ does not need to be the entire tangent bundle. However, it is usually the case in practice, and in this paper we assume  $R$  is always well-defined. Moreover, $R_x$ denotes the restriction of $R$ to $\T_x \mathcal{M}$. 
For any $x \in \mathcal{M}$, there always exists a neighborhood of $0_x$ such that the mapping $R_x$ is a diffeomorphism in the neighborhood. An important retraction is the exponential mapping, denoted by $\mathrm{Exp}$, satisfying $\mathrm{Exp}_x(\eta_x) = \gamma(1)$, where $\gamma(0) = x$, $\gamma'(0) = \eta_x$, and $\gamma$ is the geodesic passing through $x$. In a Euclidean space, the most common retraction is the exponential mapping given by addition $\Exp_x(\eta_x) = x + \eta_x$.
If the ambient space of the manifold $\mathcal{M}$ is a finite dimensional linear space, i.e., $\mathcal{M}$ is an embedded submanifold of $\mathbb{R}^n$ or a quotient manifold whose total space is an embedded submanifold of $\mathbb{R}^n$, then there exist two constants $\varkappa_1$ and $\varkappa_2$ such that the inequalities
\begin{align}
\|R_x(\eta_x) - x\| \leq \varkappa_1 \|\eta_x\|, \label{e10} \\
\|R_x(\eta_x) - x - \eta_x\| \leq \varkappa_2 \|\eta_x\|^2 \label{e11}
\end{align}
hold for any $x \in \mathcal{N}$ and $R_x(\eta_x) \in \mathcal{N}$, where $\mathcal{N}$ is a compact subset of $\mathcal{M}$.

A vector transport $\mathcal{T}: \T \mathcal{M} \oplus \T \mathcal{M} \rightarrow \T \mathcal{M}: (\eta_x, \xi_x) \mapsto \mathcal{T}_{\eta_x} \xi_x$ associated with a retraction $R$ is a smooth ($C^\infty$) mapping such that, for all $(x, \eta_x)$ in the domain of $R$ and all $\xi_x \in \T_x \mathcal{M}$, it holds that (i) $\mathcal{T}_{\eta_x} \xi_x \in \T_{R(\eta_x)} \mathcal{M}$ (ii) $\mathcal{T}_{0_x} \xi_x = \xi_x$, and (iii) $\mathcal{T}_{\eta_x}$ is a linear map. 
An important vector transport is the parallel translation, denoted  $\mathcal{P}$. The basic idea behind the parallel translation is to move a tangent vector along a given curve on a manifold ``parallelly''. We refer to~\cite{AMS2008} for its rigorous definition. 
The vector transport by differential retraction $\mathcal{T}_R$ is defined by $\mathcal{T}_{R_{\eta_x}} \xi_x = \frac{d}{d t} R_{x}(\eta_x + t \xi_x) \vert_{t = 0}$. The adjoint operator of a vector transport $\mathcal{T}$, denoted by $\mathcal{T}^\sharp$, is a vector transport satisfying $\inner[y]{\xi_y}{\mathcal{T}_{\eta_x} \zeta_x} = \inner[x]{\mathcal{T}_{\eta_x}^\sharp \xi_y}{\zeta_x}$ for all $\eta_x, \zeta_x \in \T_x \mathcal{M}$ and $\xi_y \in \T_y \mathcal{M}$, where $y = R_x(\eta_x)$. 
In the Euclidean setting, a vector transport $\mathcal{T}_{\eta_x}$ for any $\eta_x \in \T_x \mathcal{M}$ can be represented by a matrix  (the commonly-used vector transport is the identity matrix). Then the adjoint operators of a vector transport are given by the transpose of the corresponding matrix.


The Riemannian gradient of a function $h:\mathcal{M} \rightarrow \mathbb{R}$
, denote  $\grad h(x)$, is the unique tangent vector satisfying:
\begin{equation*}
\D h(x) [\eta_x] = \inner[x]{\eta_x}{\grad h(x)}, \forall \eta_x \in \T_x \mathcal{M},
\end{equation*}
where $\D h(x) [\eta_x]$ denotes the directional derivative along the direction $\eta_x$.
The Riemannian Hessian of $h$ at $x$, denoted by $\Hess h(x)$, is a linear operator on $\T_x \mathcal{M}$ satisfying
\[
\Hess h(x) [\eta_x] = \overline{\nabla}_{ \eta_x } \grad h(x), \qquad \forall \eta_x \in \T_x \mathcal{M},
\]
where $\Hess h(x) [\eta_x]$ denotes the action of $\Hess h(x)$ on a tangent vector $\eta_x \in \T_x \mathcal{M}$, and $\overline{\nabla}$ denotes the Riemannian affine connection. Roughly speaking, an affine connection generalizes the concept of a directional derivative of a vector field and we refer to ~\cite[Section~5.3]{AMS2008} for its rigorous definition. 


A function $h:\mathcal{M} \rightarrow \mathbb{R}$ is called locally Lipschitz continuous with respect to a retraction $R$ if for any compact subset $\mathcal{N}$ of $\mathcal{M}$, there exists a constant $L_h$ such that for any $x \in \mathcal{N}$ and $\xi_x, \eta_x \in \T_x \mathcal{M}$ satisfying $R_x(\xi_x) \in \mathcal{N}$ and $R_x(\eta_x) \in \mathcal{N}$, it holds that $|h \circ R (\xi_x) - h \circ R(\eta_x) | \leq L_{h} \|\xi_x - \eta_x\|$.
If $h$ is  Lipschitz continuous but not differentiable, then the Riemannian version of generalized subdifferential defined in~\cite{HHY2018} is used. Specifically, since $\hat{h}_x = h \circ R_x$ is a Lipschitz continuous function defined on a Hilbert space $\T_x \mathcal{M}$, the Clarke generalized directional derivative at $\eta_x \in \T_x \mathcal{M}$, denoted by $\hat{h}_x^\circ(\eta_x; v)$, is defined by $\hat{h}_x^\circ(\eta_x; v) = \lim_{\xi_x \rightarrow \eta_x} \sup_{t \downarrow 0} \frac{\hat{h}_x(\xi_x + t v) - \hat{h}_x(\xi_x)}{t}$, where $v \in \T_x \mathcal{M}$. The generalized subdifferential of $\hat{h}_x$ at $\eta_x$, denoted  $\partial \hat{h}_x(\eta_x)$, is defined by $\partial \hat{h}_x(\eta_x) = \{\eta_x \in \T_x \mathcal{M} \mid \inner[x]{\eta_x}{v} \leq \hat{h}_x^\circ(\eta_x; v) \hbox{ for all } v \in \T_x \mathcal{M}\}$. The Riemannian version of the Clarke generalized direction derivative of $h$ at $x$ in the direction $\eta_x \in \T_x \mathcal{M}$, denoted  $h^\circ (x; \eta_x)$, is defined by $h^\circ (x; \eta_x) = \hat{h}_x^\circ (0_x; \eta_x)$. The generalized subdifferential of $h$ at $x$, denoted  $\partial h(x)$, is defined as $\partial h(x) = \partial \hat{h}_x(0_x)$. Any tangent vector $\xi_x \in \partial h(x)$ is called a Riemannian subgradient of $h$ at $x$. 

A vector field $\eta$ is called Lipschitz continuous if there exist a positive injectivity radius $i(\mathcal{M})$ and a positive constant $L_v$ such that for all $x, y \in \mathcal{M}$ with $\dist(x, y) < i(\mathcal{M})$, it holds that
\begin{equation} \label{VFLipCon}
\| \mathcal{P}_{\gamma}^{0 \leftarrow 1} \eta_y - \eta_x \|_x \leq L_v \dist(y, x),
\end{equation}
where $\gamma$ is a geodesic with $\gamma(0) = x$ and $\gamma(1) = y$, the injectivity radius $i(\mathcal{M})$ is defined by $i(\mathcal{M}) = \inf_{x \in \mathcal{M}} i_x$ and $i_x = \sup \{\epsilon > 0 \mid \Exp_x\vert_{\mathbb{B}(x, \epsilon)} \hbox{ is a diffeomorphism} \}$. Note that for any compact manifold, the injectivity radius is positive~\cite[Lemma~6.16]{Lee2018}. A vector field $\eta$ is called locally Lipschitz continuous if for any compact subset $\bar\Omega$ of $\mathcal{M}$, there exists a positive constant $L_v$ such that for all $x, y \in \bar\Omega$ with $\dist(x, y) < i({\bar\Omega})$, inequality~\eqref{VFLipCon} holds. A function on $\mathcal{M}$ is called (locally) Lipschitz continuous differentiable if the vector field of its gradient is (locally) Lipschitz continuous.

Let $\tilde{\Omega}$ be a subset of $\mathcal{M}$. If there exists a positive constant $\varrho$  such that, for all $y \in \tilde{\Omega}, \tilde{\Omega} \subset R_y(\mathcal{B}(0_y, \varrho))$ and $R_y$ is a diffeomorphism on $\mathcal{B}(0_y, \varrho)$, then we call $\tilde{\Omega}$ a totally retractive set with respect to $\varrho$. 
The existence of $\tilde{\Omega}$ can be shown along the lines of~\cite[Theorem~3.7]{dC92}, i.e., given any $x \in \mathcal{M}$, there exists a neighborhood of $x$ which is a totally retractive set.

In a Euclidean space, the Euclidean metric is denoted by $\inner[\E]{\eta_x}{ \xi_x}$, where $\inner[\E]{\eta_x}{ \xi_x}$  is  equal to the summation of the entry-wise products of $\eta_x$ and $\xi_x$, such as $\eta_x^T \xi_x$ for vectors and $\trace(\eta_x^T \xi_x)$ for matrices. The induced Euclidean norm is denoted by $\|\cdot\|_{\mathrm{F}}$. For any matrix $M$, the spectral norm is denoted by $\|M\|_2$. For any vector $v \in \mathbb{R}^n$, the $p$-norm, denoted $\|v\|_p$, is equal to $\left(\sum_{i = 1}^n |v_i|^p \right)^{\frac{1}{p}}$.
In this paper, $\mathbb{R}^n$ does not only refer to a vector space, but also can refer to a matrix space or a tensor space.

\section{An Inexact Riemannian Proximal Gradient Method} \label{sect:IRPG}


The proposed inexact Riemannian proximal gradient method (IRPG) is stated in Algorithm~\ref{a1}. 
The search direction $\hat{\eta}_{x_k}$ at the $k$-th iteration  solves the proximal mapping 
\begin{equation} \label{e61}
\min_{\eta \in \T_{x_k} \mathcal{M}} \ell_{x_k}(\eta)  = \inner[x_k]{\grad f(x_k)}{\eta} + \frac{\tilde{L}}{2} \|\eta\|^2 + g(R_{x_k}(\eta))
\end{equation} 
approximately in the sense that its distance to a stationary point $\eta_{x_k}^*$, $\|\hat{\eta}_{x_k} - \eta_{x_k}^*\|$, is controlled from above by a continuous function $q$ of \whcomm{}{($\varepsilon_k$, $\|\hat{\eta}_{x_k}\|$)} and the function value of $\ell_{x_k}$ satisfies $\ell_{x_k}(0) \geq \ell_{x_k}(\hat{\eta}_{x_k})$. 
To the best of our knowledge, this is not Riemannian generalization of any existing Euclidean inexact proximal gradient methods. Specifically, given the exact Euclidean proximal mapping defined by $\Prox_{\lambda g}(y) = \argmin_{x} \Phi_{\lambda}(x) := \lambda g(x) + \frac{1}{2} \|x - y\|_{\F}^2$, letting $z = \Prox_{\lambda g}(y)$, it follows that $(y - z) / \lambda \in \partial^E g(z)$ and $\dist(0, \partial^E \Phi_{\lambda}(z)) = 0$, where $\partial^E$ denotes the Euclidean subdifferential. Based on these observations, the inexact Euclidean proximal mappings proposed in~\cite{Rock1976,SRB2011,VSBV2013,BPR2020} only require  $z$ to satisfy \kwrev{any one of the following conditions:}
\begin{equation} \label{e57}
\mathrm{dist}(0, \partial^E \Phi_{\lambda}(z)) \leq \frac{\varepsilon}{\lambda}, \quad \Phi_{\lambda}(z) \leq \min \Phi_{\lambda} + \frac{\varepsilon^2}{2 \lambda}, \hbox{ and } \frac{y - z}{\lambda} \in \partial_{\frac{\varepsilon^2}{2 \lambda}}^E g(z),
\end{equation}
where $\partial_{\epsilon}^E$ denotes the Euclidean $\epsilon$-subdifferential. 
The corresponding analyses and algorithms rely on the properties of $\epsilon$-subdifferential of convex functions. However, since the function \kwrev{$\ell_{x_k}(\eta)$} is not necessarily convex, these techniques cannot be applied.
\whcomm{}{Note that if $g$ is convex and  $\mathcal{M}$ is a Euclidean space, then the function $\Phi$ is strongly convex. Therefore, the solutions of the inexact Euclidean proximal mappings in~\eqref{e57} all satisfy~\eqref{e01} with certain $q$ by choosing an appropriate choice of $\varepsilon$.}

\begin{algorithm}[ht!]
\caption{An Inexact Riemannian Proximal Gradient Method (IRPG)}
\label{a1}
\begin{algorithmic}[1]
\Require Initial iterate $x_0$; a sufficiently large positive constant $\tilde{L}$;
\For {$k = 0, 1, \ldots$}
\State Let $\ell_{x_k}(\eta) = \inner[x_k]{\grad f(x_k)}{\eta} + \frac{\tilde{L}}{2} \|\eta\|^2 + g(R_{x_k}(\eta))$;
\State \label{a1:st1}
Find $\hat{\eta}_{x_k}{\in\T_{x_k}\mathcal{M}}$ such that the following two conditions hold
\begin{align} 
\|\hat{\eta}_{x_k} - \eta_{x_k}^*\| \leq q(\whcomm{}{\varepsilon_k}, \|\hat{\eta}_{x_k}\|) \hbox{ and } \ell_{x_k}(0) \geq \ell_{x_k}(\hat{\eta}_{x_k}), \label{e01} 
\end{align}
where 
$\varepsilon_k > 0$, 
and $q:\mathbb{R}^2 \rightarrow \mathbb{R}$ is a continuous function;
\State \label{a1:st2} $x_{k+1} = R_{x_k}(\hat{\eta}_{x_k})$;
\EndFor
\end{algorithmic}
\end{algorithm}

\whnew{}{
In Algorithm~\ref{a1}, $q$ controls the accuracy for solving the  proximal mapping and different accuracies lead to different convergence results. Here we give four choices of $q$:
\begin{enumerate}
    \item[1)] $q(\varepsilon_k, \|\hat{\eta}_{x_k}\|) = \varepsilon_k$ with $\varepsilon_k \rightarrow 0$;
    \item[2)] $q(\varepsilon_k, \|\hat{\eta}_{x_k}\|) = \tilde{q}(\|\hat{\eta}_{x_k}\|)$ with $\tilde{q}: \mathbb{R} \rightarrow [0, \infty)$ a continuous function satisfying $\tilde{q}(0) = 0$;
    \item[3)] $q(\varepsilon_k, \|\hat{\eta}_{x_k}\|) = \varepsilon_k^2$, with $\sum_{k = 0}^{\infty} \varepsilon_k < \infty$; and
    \item[4)] $q(\varepsilon_k, \|\hat{\eta}_{x_k}\|) = \min(\varepsilon_k^2, \delta_q \|\hat{\eta}_{x_k}\|^2)$ with a constant $\delta_q > 0$ and $\sum_{k = 0}^{\infty} \varepsilon_k < \infty$.
\end{enumerate}
The  four choices all satisfy the requirement for the global convergence in Theorem~\ref{globaltheo}, with the first one being the weakest. A practical scheme discussed in Section~\ref{sect:subprobGlobal} can yield a $\hat{\eta}_{x_k}$ that satisfies the second choice. The third $q$ guarantees that the accumulation point is unique as shown in Theorem~\ref{single}. The last $q$ allows us to establish convergence rate analysis of Algorithm~\ref{a1} based on the Riemannian KL property, as shown in Theorem~\ref{LocalRateKL}. The practical scheme for generating $\hat{\eta}_{x_k}$ that satisfies the third and fourth choices is discussed in Section~\ref{sec:LocalRate}.}

\subsection{Global Convergence Analysis}\label{subsec:global}

The global convergence analysis is over similar to that in~\cite{HuaWei2019b} 
and relies on Assumptions~\ref{as10} and~\ref{as3} below. Assumption~\ref{as10} is mild in the sense that it holds if the manifold $\mathcal{M}$ is compact and the function $F$ is continuous.
\begin{assumption} \label{as10}
The function $F$ is bounded from below and the sublevel set $\Omega_{x_0} = \{x \in \mathcal{M} \mid F(x) \leq F(x_0)\}$ is compact.
\end{assumption}

Definition~\ref{def:Lsmooth} has been used in~\cite{BAC2018,HuaWei2019b}. It generalizes the $L$-smoothness from the Euclidean setting to the Riemannian setting. It says that if the composition $h \circ R$ satisfies the Euclidean version of $L$-smoothness, then $h$ is called a $L$-retraction-smooth function.

\begin{definition} \label{def:Lsmooth}
A function $h:\mathcal{M} \rightarrow \mathbb{R}$ is called $L$-retraction-smooth with respect to a retraction $R$ in $\mathcal{N} \subseteq \mathcal{M}$ if for any $x \in \mathcal{N}$ and any $\mathcal{S}_x \subseteq \T_x \mathcal{M}$ such that $R_x(\mathcal{S}_x)\subseteq \mathcal{N}$, we have that 
\begin{equation} \label{e65}
h (R_x(\eta)) \leq h(x) + \inner[x]{\grad h(x)}{\eta} + \frac{L}{2} \|\eta\|^2,\quad \forall \eta \in \mathcal{S}_x.
\end{equation}
\end{definition}

 In Assumption~\ref{as3}, we assume that the function $f$ is $L$-smooth in the sublevel set $\Omega_{x_0}$. This  is also mild and practical methods to verify this assumption have been given in~\cite[Lemma~2.7]{BAC2018}.

\begin{assumption} \label{as3}
The function $f$ is $L$-retraction-smooth with respect to the retraction $R$ in the sublevel set $\Omega_{x_0}$.
\end{assumption}

Lemma~\ref{le3} shows that IRPG is a descent algorithm. The short proof is the same as that for \cite[Lemma~1]{HuaWei2019b}, but  included for completeness. \slversions{Its proof follows the same derivation of~\cite[Lemma~1]{HuaWei2019b} and therefore is omitted here.}{}
\begin{lemma} \label{le3}
Suppose Assumption~\ref{as3} holds and $\tilde{L} > L$. Then the sequence $\{x_k\}$ generated by Algorithm~\ref{a1} satisfies
\begin{align}
F(x_k) - F(x_{k+1}) \geq& \beta \|\hat{\eta}_{x_k}\|^2, \label{e9} 
\end{align}
where $\beta = (\tilde{L} - L) / 2$.
\end{lemma}
\slversions{}{
\begin{proof}
By the definition of $\hat{\eta}_{x_k}$ and the $L$-retraction-smooth of $f$, we have
\begin{align}
F(x_k) &= f(x_k) + g(x_k) \geq f(x_k) + \inner[x_k]{\grad f(x_k)}{\hat{\eta}_{x_k}} + \frac{\tilde{L}}{2} \|\hat{\eta}_{x_k}\|^2 + g(R_{x_k}(\hat{\eta}_{x_k})) \nonumber \\
&\geq \frac{\tilde{L} - L}{2} \|\hat{\eta}_{x_k}\|^2 + f(R_{x_k}(\hat{\eta}_{x_k})) + g(R_{x_k}(\hat{\eta}_{x_k})) = F(x_{k+1}) + \frac{\tilde{L} - L}{2} \|\hat{\eta}_{x_k}\|^2, \nonumber
\end{align}
which completes the first result. 
\end{proof}
}

Now, we are ready to give a global convergence analysis of IRPG. 
\begin{theorem} \label{globaltheo}
Suppose that Assumptions~\ref{as10} and~\ref{as3} hold, that $\tilde{L} > L$, and that \whnew{}{$\lim_{k \rightarrow \infty} q(\varepsilon_k, \|\hat{\eta}_{x_k}\|) = 0$.}
Then the sequence $\{x_k\}$ has at least one accumulation point.
Let $x_*$ be any accumulation point of the sequence $\{x_k\}$.
Then $x_*$ is a stationary point.
Furthermore, Algorithm~\ref{a1} returns $x_k$ satisfying $\|\hat{\eta}_{x_k}\| \leq \epsilon$ in at most $(F(x_0) - F(x_*)) / (\beta \epsilon^2)$ iterations.
\end{theorem}
\begin{proof}
The proof mainly follows the proof in~\cite[Theorem~1]{HuaWei2019b}. Here, we only highlight the differences.
The existence of an accumulation point follows immediately from Assumption~\ref{as10} and Lemma~\ref{le3}.

By Lemma~\ref{le3}, we have that $F(x_0) - F(\tilde{x}) \geq \beta \sum_{i = 0}^\infty \|\hat{\eta}_{x_k}\|^2$, where $\tilde{x}$ denotes a global minimizer of $F$. Therefore, 
\begin{equation} \label{e79}
\lim_{k \rightarrow \infty} \|\hat{\eta}_{x_k}\| = 0.
\end{equation}
By~\whnew{}{$\lim_{k \rightarrow \infty} q(\varepsilon_k, \|\hat{\eta}_{x_k}\|) = 0$}, we have
\begin{equation*}
\lim_{k \rightarrow \infty} \|\eta_{x_k}^*\| = 0.
\end{equation*}
The remaining of the proof follows~\cite[Theorem~1]{HuaWei2019b}.
\end{proof}

\subsection{Local Convergence Rate Analysis Using Riemannian Kurdyka-\L ojasiewicz Property}\label{subsec:local}

The KL property has been widely used for the convergence analysis of various convex and nonconvex algorithms in the Euclidean case, see e.g., \cite{ABRS2010,ABS2013,BST2014,LL2015}. In this section we will study the convergence of RPG base on the Riemannian Kurdyka-\L ojasiewicz (KL) property, introduced in~{\cite{KMP2000}} for the analytic setting and in~\cite{BCO2011} for the nonsmooth setting. \whcomm{}{Note that a convergence analysis based on KL property for a Euclidean inexact proximal gradient has been given in~\cite{BPR2020}. As we pointed out before, the convergence analysis and algorithm  design therein rely on the convexity of the objective in the proximal mapping.}

\begin{definition} \label{def:RKL}
A continuous function $f: \mathcal{M} \rightarrow \mathbb{R}$ is said to have the Riemannian KL property at $x \in \mathcal{M}$ if and only if there exists $\varepsilon \in (0, \infty]$, a neighborhood $U \subset \mathcal{M}$ of $x$, and a continuous concave function $\varsigma: [0, \varepsilon] \rightarrow [0, \infty)$ such that
\begin{itemize}
\item $\varsigma(0) = 0$,
\item $\varsigma$ is $C^1$ on $(0, \varepsilon)$,
\item $\varsigma' > 0$ on $(0, \varepsilon)$,
\item For every $y \in U$ with $f(x) < f(y) < f(x) + \varepsilon$, we have 
\[
\varsigma'(f(y) - f(x)) \dist(0, \partial f(y)) \geq 1,
\]
where $\dist(0, \partial f(y)) = \inf\{ \|v\|_y : v \in \partial f(y) \}$ and $\partial$ denotes the Riemannian generalized subdifferential. The function $\varsigma$ is called the desingularising function.
\end{itemize}
\end{definition}
\slversions{}{Note that the definition of the Riemannian KL property is overall similar to the KL property in the Euclidean setting, except that related notions including $U$, $\partial f(y)$ and $\dist(0, \partial f(y))$ are all defined on a manifold.}
In~\cite{BCO2011,HuaWei2019b}, sufficient conditions to verify if a function satisfies the Riemannian KL condition are given.

Assumptions~\ref{as19} and~\ref{as04} are used for the convergence analysis in this section. Assumption~\ref{as19} is a standard assumption and has been made in e.g.,~\cite{LL2015}, when the manifold $\mathcal{M}$ is the Euclidean space.
\begin{assumption} \label{as19}
The function $f:\mathcal{M} \rightarrow \mathbb{R}$ is locally Lipschitz continuously differentiable. 
\end{assumption}

\begin{assumption} \label{as04}
The function $F$ is locally Lipschitz continuous with respect to the retraction $R$.
\end{assumption}



In order to guarantee the uniqueness of accumulation points, the Riemannian proximal mapping needs to be solved more accurately than~\eqref{e01}, as shown in Theorem~\ref{single}. 

\begin{theorem}\label{single}
Let $\{x_k\}$ denote the sequence generated by Algorithm~\ref{a1} and $\mathcal{S}$ denote the set of all accumulation points.
Suppose Assumptions~\ref{as10}, ~\ref{as3}, ~\ref{as19} and~\ref{as04} hold. 
We further assume that $F = f + g$ satisfies the Riemannian KL property at every point in $\mathcal{S}$. 
If the Riemannian proximal mapping~\eqref{e23} is solved such that 
for all $k$,
\begin{align} \label{e28}
\|\hat{\eta}_{x_k} - {\eta}_{x_k}^*\| \leq \whnew{}{\varepsilon_k^2}, 
\end{align}
\whcomm{}{that is, $q(\varepsilon_k, \|\hat{\eta}_{x_k}\|) = \varepsilon_k^2$.}
Then, 
\begin{align*}
\sum_{k=0}^\infty \dist(x_k,x_{k+1})<\infty. \numberthis\label{eq:finite}
\end{align*}
It follows that $\mathcal{S}$ only contains a single point.
\end{theorem}
\begin{proof}
First note that the global convergence result in Theorem~\ref{globaltheo} implies that every point in $\mathcal{S}$
is a stationary point. Since  $\lim_{k \rightarrow \infty} \|\hat{\eta}_{x_k}\| = 0$, there exists a $\delta_T>0$ such that $\|\hat{\eta}_{x_k}\| \leq \delta_T$ for all $k$. Thus, the application of~\cite[Lemma~6]{HuaWei2019b} implies that
\begin{align*}
\dist(x_{k},x_{k+1})=\dist(x_k,R_{x_k}(\hat{\eta}_{x_k}))\leq\kappa \|\hat{\eta}_{x_k}\| \rightarrow 0.\numberthis\label{eq:kwrev01a}
\end{align*}
Then by \cite[Remark 5]{BST2014}, we know that $\mathcal{S}$ is a compact set.
Moreover, since $F(x_k)$ is nonincreasing and $F$ is continuous, $F$ has the same value at all the points in~$\mathcal{S}$.  Therefore, by~\cite[Lemma~5]{HuaWei2019b}, there exists a single  desingularising function, denoted $\varsigma$, for the Riemannian KL property  of $F$ to hold at all the points in $\mathcal{S}$.

Let $x_*$ be a point in $\mathcal{S}$. Assume there exists $\bar k$ such that $x_{\bar k}=x_*$. Since $F(x_k)$ is non-increasing, it must hold $F(x_{\bar k})=F(x_{{\bar k}+1})$. By Lemma~\ref{le3}, we have $\eta_{x_{\bar k}}^*=0$,  $x_{\bar k}=x_{{\bar k}+1}$, \eqref{eq:finite} holds evidently. 

In the case when $F(x_k)>F(x_*)$ for all $k$, Since $\eta_{x_k}^* \rightarrow 0$, we have $F(R_{x_k}(\eta_{x_k}^*)) \rightarrow F(x_*)$, $\dist(R_{x_k}(\eta_{x_k}^*), \mathcal{S}) \rightarrow 0$. By the Riemannian KL property of $F$ on $\mathcal{S}$, there exists an $l>0$  such that 
\begin{align*}
\varsigma'(F(R_{x_k}(\eta_{x_k}^*)) - F(x_*)) \dist(0, \partial F(R_{x_k}(\eta_{x_k}^*))) \geq 1\quad\mbox{for all }k>l. 
\end{align*}
It follows that 
\begin{align*}
\varsigma'(F(R_{x_k}(\eta_{x_k}^*)) - F(x_*))\geq \dist(0, \partial F(R_{x_k}(\eta_{x_k}^*)))^{-1}\quad\mbox{for all }k>l.\numberthis\label{eq:kwrev01}
\end{align*}
Since $\lim_{k \rightarrow \infty} \|\eta_{x_{k}}^*\| = 0$, there exists a constant $k_0 > 0$ such that $\|\eta_{x_{k}}^*\| < \mu$ for all $k > k_0$, where $\mu$ is defined in \cite[Lemma~7]{HuaWei2019b}. By Assumption~\ref{as19} and~\cite[Lemma~7]{HuaWei2019b}, we have
\begin{align} \label{e96}
\|\grad f(R_{x_{k}}(\eta_{x_{k}}^*)) - \mathcal{T}_{R_{\eta_{x_{k}}^*}}^{- \sharp} (\grad f(x_{k}) + \tilde{L} \eta_{x_{k}}^*)\| \leq L_c \|\eta_{x_{k}}^*\|
\end{align}
for all $k \geq k_0$, where $L_c$ is a constant. \kwrev{By the definition of $\eta_{x_{k}}^*$, there exists $\zeta_{x_{k}} \in \partial g( R_{x_{k}} (\eta_{x_{k}}^*) )$ such that
\begin{equation} 
	\grad f(x_{k}) + \tilde{L} \eta_{x_{k}}^* + \mathcal{T}_{R_{\eta_{x_{k}}^*}}^\sharp \zeta_{x_{k}} = 0.
\end{equation} }
It follows that
\begin{align}
	\grad f(R_{x_{k}}(\eta_{x_{k}}^*)) -& \mathcal{T}_{R_{\eta_{x_{k}}^*}}^{- \sharp} (\grad f(x_{k}) + \tilde{L}  \eta_{x_{k}}^*) \nonumber \\
	&= \grad f(R_{x_{k}}(\eta_{x_{k}}^*)) + \zeta_{x_{k}} \in \partial F(R_{x_{k}}(\eta_{x_{k}}^*)).  \label{e97}
\end{align}
Therefore,~\eqref{e96} and~\eqref{e97} yield
\begin{equation} \label{e66}
\dist(0, \partial F(R_{x_{k}}(\eta_{x_{k}}^*))) \leq L_c \|\eta_{x_{k}}^*\|,
\end{equation}
for all $k > k_0$. Inserting this into \eqref{eq:kwrev01} gives 
\begin{align*}
\varsigma'(F(R_{x_k}(\eta_{x_k}^*)) - F(x_*))\geq L_c^{-1}\|\eta_{x_{k}}^*\|^{-1}\quad\mbox{for all }k>\hat{l}:=\max(k_0,l).\numberthis\label{eq:kerev02}
\end{align*}
Define $\Delta_{p, q} = \varsigma(F(x_{p})-F(x_*))-\varsigma(F(x_{q})-F(x_*))$. \whnew{[WH: I updated the proofs below slightly.]}{} We next show that for sufficiently large $k$,
\begin{equation} \label{e06}
\|\hat{\eta}_{x_{k}}\|^2 \leq b_0 \Delta_{k, k+1} (\|\hat{\eta}_{x_{k-1}}\| + \varepsilon_{k-1}^2) + b_1 \varepsilon_{k-1}^2,
\end{equation}
where $b_0 = \frac{L_c}{\beta}$, $b_1 = \frac{L_F}{\beta}$, and $L_F$ is the Lipschitz constant of $F$. To the end, we consider two cases:
\begin{itemize}
\item Case 1: $F(x_k) = F(R_{x_{k-1}}(\hat{\eta}_{x_{k-1}})) \leq F(R_{x_{k-1}}(\eta_{x_{k-1}}^*))$. \\
We have
\begin{align*}
\varsigma(F(x_{k})-F(x_*))-&\varsigma(F(x_{k+1})-F(x_*)) \\
&\geq \varsigma'(F(x_{k})-F(x_*))(F(x_{k})-F(x_{k+1})) \\
&\geq \varsigma'(F(R_{x_{k-1}}(\eta_{x_{k-1}}^*))-F(x_*))(F(x_{k})-F(x_{k+1}))\\
&\geq L_c^{-1}\beta \frac{\|\hat{\eta}_{x_{k}}\|^2}{\|\eta_{x_{k-1}}^*\|} \\
&\geq \frac{\beta}{L_c} \frac{\|\hat{\eta}_{x_{k}}\|^2}{\whnew{}{(\|\hat{\eta}_{x_{k-1}}\| + \varepsilon_{k-1}^2)}} \quad\mbox{for all }k>\hat{l}:=\max(k_0,l), 
\end{align*}
where the first and the second inequalities are from the concavity of $\varsigma$, the third inequality is from~\eqref{eq:kerev02} and~\eqref{e9}, and the last inequality is from~\eqref{e28}. It follows that
\[
\|\hat{\eta}_{x_{k}}\|^2 \leq \frac{L_c}{\beta} \Delta_{k, k+1} (\|\hat{\eta}_{x_{k-1}}\| + \varepsilon_{k-1}^2),
\]
which implies that~\eqref{e06} holds.
\item Case 2: $F(x_k) = F(R_{x_{k-1}}(\hat{\eta}_{x_{k-1}})) > F(R_{x_{k-1}}(\eta_{x_{k-1}}^*))$. \\
We have
\begin{align}
&\qquad \varsigma(F(x_{k})-F(x_*))-\varsigma(F(x_{k+1})-F(x_*)) \nonumber \\
&\geq \varsigma(F(R_{x_{k-1}}(\eta_{x_{k-1}}^*))-F(x_*))-\varsigma(F(x_{k+1})-F(x_*)) \nonumber \\
&\geq \varsigma'(F(R_{x_{k-1}}(\eta_{x_{k-1}}^*)) - F(x_*))(F(R_{x_{k-1}}(\eta_{x_{k-1}}^*))-F(x_{k+1})) \nonumber \\
&= \varsigma'(F(R_{x_{k-1}}(\eta_{x_{k-1}}^*))-F(x_*)) \Big(F(R_{x_{k-1}}(\eta_{x_{k-1}}^*))  \nonumber\\
&\qquad \qquad - F(R_{x_{k-1}}(\hat{\eta}_{x_{k-1}})) + F(x_k) -F(x_{k+1}) \Big) \nonumber \\
&\geq \frac{\beta \|\hat{\eta}_{x_{k}}\|^2 - L_F \|{\eta}_{x_{k-1}}^* - \hat{\eta}_{x_{k-1}}\|}{L_c\|\eta_{x_{k-1}}^*\|} \quad\mbox{for all }k>\hat{l}:=\max(k_0,l) \nonumber \\
&\geq \frac{\beta \|\hat{\eta}_{x_{k}}\|^2 - L_F \varepsilon_{k-1}^2}{L_c(\|\hat{\eta}_{x_{k-1}}\| + \varepsilon_{k-1}^2)}, \label{e03}
\end{align}
where the third inequality is from Assumption~\ref{as04} with Lipschitz constant denoted by $L_F$ and the last inequality is from~\eqref{e28}.
Together with~\eqref{e28}, inequality~\eqref{e03} yields that for all $k > \hat{\ell}$,
\begin{align} \label{e04}
\beta \|\hat{\eta}_{x_{k}}\|^2 \leq L_c \Delta_{k, k+1} (\|\hat{\eta}_{x_{k-1}}\| + \varepsilon_{k-1}^2) + L_F \varepsilon_{k-1}^2,
\end{align}
which gives
\begin{equation} \label{e05}
\|\hat{\eta}_{x_{k}}\|^2 \leq \frac{L_c}{\beta} \Delta_{k, k+1} (\|\hat{\eta}_{x_{k-1}}\| + \varepsilon_{k-1}^2) + \frac{L_F}{\beta} \varepsilon_{k-1}^2.
\end{equation}
which implies that~\eqref{e06} holds.
\end{itemize}
Once \eqref{e06} has been established, by $\sqrt{a^2 + b^2} \leq a + b$ and $2 \sqrt{a b} \leq a + b$, we have
\begin{equation} \label{e07}
2 \|\hat{\eta}_{x_{k}}\| \leq  b_0 \Delta_{k, k+1} + \|\hat{\eta}_{x_{k-1}}\| + \varepsilon_{k-1}^2 + 2 \sqrt{b_1} \varepsilon_{k-1}. 
\end{equation}
For any $p > \hat{l}$, taking summation of~\eqref{e07} from $p$ to $s$ yields
\begin{align*}
\sum_{k = p}^s 2 \|\hat{\eta}_{x_{k}}\|  \leq \sum_{k = p}^s b_0 \Delta_{k, k+1} + \sum_{k = p}^s \|\hat{\eta}_{x_{k-1}}\| + 2 \sqrt{b_1} \sum_{k = p}^s \varepsilon_{k-1} + \sum_{k=p}^s \varepsilon_{k-1}^2, 
\end{align*}
which implies
\begin{equation*}
\sum_{k = p}^s \|\hat{\eta}_{x_{k}}\| \leq \|\hat{\eta}_{x_{p-1}}\| + b_0 \Delta_{p, s+1} + 2 \sqrt{b_1} \sum_{k = p}^s \varepsilon_{k-1} + \sum_{k=p}^s \varepsilon_{k-1}^2.
\end{equation*}
Taking $s$ to $\infty$ yields
\begin{equation} \label{e08}
\sum_{k = p}^{\infty} \|\hat{\eta}_{x_{k}}\| \leq \|\hat{\eta}_{x_{p-1}}\| + b_0 \varsigma(F(x_{p})-F(x_*)) + 2 \sqrt{b_1} \sum_{k = p}^{\infty} \varepsilon_{k-1} + \sum_{k=p}^{\infty} \varepsilon_{k-1}^2.
\end{equation}
It follows that $\sum_{k=0}^\infty\|\hat{\eta}_{x_{k}}\|<\infty$, which yields~\eqref{eq:finite} due to \eqref{eq:kwrev01a}.
\end{proof}

Theorem~\ref{LocalRateKL} gives the local convergence rate based on the Riemannian KL property. Note that the local convergence rate requires an even more accurate solution than that in Theorem~\ref{single}. 
\begin{theorem} \label{LocalRateKL}
Let $\{x_k\}$ denote the sequence generated by Algorithm~\ref{a1} and $\mathcal{S}$ denote the set of all accumulation points.
Suppose Assumptions~\ref{as10}, ~\ref{as3}, ~\ref{as19}, and~\ref{as04} hold.  
We further assume that $F = f + g$ satisfies the Riemannian KL property at every point in $\mathcal{S}$ with the desingularising function having the form of $\varsigma(t) = \frac{C}{\theta} t^{\theta}$ for some $C > 0$, $\theta \in (0, 1]$. The accumulation point, denoted $x_*$, is unique by Theorem~\ref{single}. 
If the Riemannian proximal mapping~\eqref{e23} is solved such that 
for all $k$, \whnew{[WH: below $q$ and corresponding proofs are  modified.]}{}
\begin{align} \label{e29}
\|\hat{\eta}_{x_k} - {\eta}_{x_k}^*\| \leq \min\left( \varepsilon_k^2, \frac{\beta}{2 L_F} \|\hat{\eta}_{x_k}\|^2 \right),
\end{align}
\whcomm{}{that is, $q(\varepsilon_k, \|\hat{\eta}_{x_k}\|) = \min\left( \varepsilon_k^2, \frac{\beta}{2 L_F} \|\hat{\eta}_{x_k}\|^2 \right)$.}
Then
\begin{itemize}
\item \whcomm{}{If $\theta = 1$, then there exists $k_1$ such that $x_k = x_*$ for all $k > k_1$.}
\item if $\theta \in [\frac{1}{2}, 1)$, then there exist constants $C_r > 0$ and $d \in (0, 1)$ such that for all~$k$
\[
\dist(x_k, x_*) < C_r d^{k};
\]
\item if $\theta \in (0, \frac{1}{2})$, then there exists a positive constant $\tilde{C}_r$ such that for all $k$
\[
\dist(x_k, x_*) < \tilde{C}_r k^{\frac{-1}{1 - 2 \theta}}.
\]
\end{itemize}
\end{theorem}
\begin{proof}
\whcomm{}{In the case of $\theta = 1$, suppose $F(x_k) > F(x_*)$. It follows from~\eqref{eq:kwrev01} and~\eqref{e66} that
\begin{align*}
\dist(0, \partial F(R_{x_k}(\eta_{x_k}^*))) &\geq C \quad \mbox{for all }k>l, \\
\dist(0, \partial F(R_{x_{k}}(\eta_{x_{k}}^*))) &\leq L_c \|\eta_{x_{k}}^*\| \quad \mbox{for all } k > k_0.
\end{align*} 
Therefore, we have $\|\eta_{x_k}^*\| \geq C / L_c \quad \hbox{for all } k > \max(k_0, l)$.
By~\eqref{e01}, there exists $k_2>0$ and $\hat{C} > 0$ such that $\|\hat{\eta}_{x_k}\| \geq \hat{C} \|\eta_{x_k}^*\|$ for all $k \geq k_2$. It follows that
\[
\|\hat{\eta}_{x_k}\| \geq \hat{C} C / L_c, \quad \hbox{for all } k > \max(k_0, k_2, l).
\]
Due to the descent property in~\eqref{e9}, there must exist $k_1$ such that $x_k = x_*$ for all $k > k_1$.
}

\whcomm{}{Next, we consider $\theta \in (0, 1)$.}
By the same derivation as the proof in Theorem~\ref{single} and noting the difference between~\eqref{e29} and~\eqref{e28}, we obtain from~\eqref{e06} that
\[
\|\hat{\eta}_{x_{k}}\|^2 \leq b_0 \Delta_{k, k+1} (\|\hat{\eta}_{x_{k-1}}\| + \frac{\beta}{2 L_F} \|\hat{\eta}_{x_k}\|^2) + b_1 \frac{\beta}{2 L_F} \|\hat{\eta}_{x_k}\|^2,
\]
by replacing $\varepsilon_{k-1}$ with $\frac{\beta}{2 L_F} \|\hat{\eta}_{x_k}\|$.
Since $\|\hat{\eta}_{x_k}\| \rightarrow 0$, for any $\delta > 0$, there exists $k_2 > 0$ such that for all $k> k_2$, it holds that $1 + \beta / (2 L_F) \|\hat{\eta}_{x_k}\| < \delta$. Therefore, we have
\[
\|\hat{\eta}_{x_{k}}\|^2 \leq b_0 \Delta_{k, k+1} (1 + \delta) \|\hat{\eta}_{x_{k-1}}\| + \frac{1}{2} \|\hat{\eta}_{x_k}\|^2.
\]
By $2\sqrt{a b} \leq a + b$, we have
\[
2 \|\hat{\eta}_{x_{k}}\| \leq \tilde{b}_0 \Delta_{k, k+1}  + \|\hat{\eta}_{x_{k-1}}\|
\]
where $\tilde{b}_0 = 2 b_0 (1 + \delta)$.
It follows that
\begin{equation} \label{e31}
\sum_{k = p}^{\infty} \|\hat{\eta}_{x_{k}}\| \leq \|\hat{\eta}_{x_{p-1}}\| + \tilde{b}_0 \varsigma(F(x_{p})-F(x_*)).
\end{equation}
Substituting $\varsigma(t) = \frac{C}{\theta} t^{\theta}$ into~\eqref{e31} yields
\begin{equation} \label{e17}
\sum_{k = p}^{\infty} \|\hat{\eta}_{x_{k}}\| \leq \|\hat{\eta}_{x_{p-1}}\| + \frac{\tilde{b}_0 C}{\theta} (F(x_{p})-F(x_*))^{\theta}.
\end{equation}
By Assumption~\ref{as04} and~\eqref{e29}, we have
\begin{align}
|F(x_{p})-F( R_{x_{p-1}} ( {\eta}_{x_{p-1}}^* ))| =& |F( R_{x_{p-1}} ( \hat{\eta}_{x_{p-1}} )) - F( R_{x_{p-1}} ( {\eta}_{x_{p-1}}^* ))| \nonumber \\
\leq& L_F \| \hat{\eta}_{x_{p-1}} - {\eta}_{x_{p-1}}^* \| \leq \frac{\beta}{2} \|\hat{\eta}_{x_{p-1}}\|^2. \label{e18}
\end{align}
Combining~\eqref{e17} and~\eqref{e18} yields
\begin{equation} \label{e103}
\sum_{k = p}^{\infty} \|\hat{\eta}_{x_k}\| \leq \|\hat{\eta}_{x_{p-1}}\| + \frac{\tilde{b}_0 C}{\theta} \left( F( R_{x_{p-1}} ( {\eta}_{x_{p-1}}^* ) ) - F(x_*) + \frac{\beta}{2} \|\hat{\eta}_{x_{p-1}}\|^2 \right)^{\theta}.
\end{equation}

By~\eqref{eq:kwrev01}, we have $ \frac{1}{C} (F( R_{x_{p-1}} ( {\eta}_{x_{p-1}}^* ) ) - F(x_*))^{1 - \theta} \leq \dist(0, \partial F( R_{x_{p-1}} ( {\eta}_{x_{p-1}}^* ) ) $. Combining this inequality with~\eqref{e66} yields 
\begin{equation} \label{e102}
\frac{1}{C} (F( R_{x_{p-1}} ( {\eta}_{x_{p-1}}^* ) ) - F(x_*))^{1 - \theta} \leq L_c \|{\eta}_{x_{p-1}}^*\|.
\end{equation}
It follows from~\eqref{e103} and~\eqref{e102} that
\begin{align}
\sum_{k = p}^{\infty} \|\hat{\eta}_{x_k}\| \leq& \|\hat{\eta}_{x_{p-1}}\| + \frac{\tilde{b} C}{\theta} \left( (C L_c  \|{\eta}_{x_{p-1}}^*\|)^{\frac{1}{1 - \theta}} + \frac{\beta}{2} \|\hat{\eta}_{x_{p-1}}\|^2 \right)^{\theta} \nonumber  \\
\leq& \|\hat{\eta}_{x_{p-1}}\| + \frac{\tilde{b}_0 C}{\theta} \left( \left(C L_c (1+\delta)  \|\hat{\eta}_{x_{p-1}}\|\right)^{\frac{1}{1 - \theta}} + \frac{\beta}{2} \|\hat{\eta}_{x_{p-1}}\|^2 \right)^{\theta}. \label{e104}
\end{align}
Since $\lim_{k\rightarrow \infty} \|\hat{\eta}_{x_k}\| = 0$, there exists $\hat{p} > 0$ such that $\|\hat{\eta}_{x_k}\| < 1$ for all $k > \hat{p}$. Therefore, for all $p > \hat{p}$, it holds that
\[
\|\hat{\eta}_{x_{p-1}}\|^{\frac{1}{1 - \theta}} \leq \|\hat{\eta}_{x_{p-1}}\|^{ \min\left(2, \frac{1}{1 - \theta} \right) } \hbox{ and } \|\hat{\eta}_{x_{p-1}}\|^2 \leq \|\hat{\eta}_{x_{p-1}}\|^{ \min\left(2, \frac{1}{1 - \theta} \right) }
\]
which combining with~\eqref{e104} yields
\[
\sum_{k = p}^{\infty} \|\hat{\eta}_{x_k}\| \leq \|\hat{\eta}_{x_{p-1}}\| + \frac{\tilde{b}_0 C}{\theta} \left( \left( C L_c (1 + \delta)  \right)^{\frac{1}{1-\theta}} + \beta / 2  \right)^{\theta} \|\hat{\eta}_{x_{p-1}}\|^{ \min\left(2 \theta, \frac{\theta}{1 - \theta} \right) }.
\]
Note that if \kwrev{$\theta \geq 0.5$, then $1\leq 2 \theta \leq \theta / (1 - \theta)$. Thus $ \min\left(2 \theta, \frac{\theta}{1 - \theta} \right) \geq 1$. If $\theta \in (0, 0.5)$, then $\min\left(2 \theta, \frac{\theta}{1 - \theta} \right)= \theta / (1 - \theta) <1$.} The remaining part of the proof follow the same derivation as those  in~\cite[Appendix~B]{HuaWei2019extended} and~\cite[Theorem~2]{AB2009}.
\end{proof}

\section{Practical Conditions for Solving Riemannian Proximal Mapping} \label{sect:subprob}
\mbox{}
\kwrev{In the general framework of the inexact RPG method (i.e., Algorithm~\ref{a1}), the required accuracy for solving the Riemannian proximal mapping involves  the unknown exact solution~$\eta_{x_k}^*$}. In this section, we study two practical conditions that can generate  search directions satisfying~\eqref{e01} \kwrev{for different forms of $g$} when the manifold $\mathcal{M}$ has a linear ambient space, or equivalently, $\mathcal{M}$ is an embedded submanifold of $\mathbb{R}^n$ or a quotient manifold whose total space is an embedded submanifold of $\mathbb{R}^n$. Throughout this section, the Riemannian metric is fixed to be the Euclidean metric $\inner[\F]{}{}$.
We describe the algorithms for embedded submanifolds and point out here that in the case of a quotient manifold, the derivations still hold by  replacing the tangent space $\T_x \mathcal{M}$ with the notion of horizontal space $\mathcal{H}_x$.

\subsection{\kwrev{Practical Condition that Ensures Global Convergence} } \label{sect:subprobGlobal}
\kwrev{We first show that an approximate solution to the Riemannian proximal mapping in~\cite{CMSZ2019}
satisfies the condition that is needed to establish the global convergence of IRPG.} Recall that the Riemannian proximal mapping therein is
\begin{equation} \label{e34}
\tilde{\eta}_x = \argmin_{\eta \in \T_x \mathcal{M}} \tilde{\ell}_x(\eta) = \langle \grad f(x),\eta \rangle_{\F}+\frac{\tilde{L}}{2}\langle \eta, \eta\rangle_{\F} + g(x+\eta).
\end{equation}
Since $\mathcal{M}$ has a linear ambient space $\mathbb{R}^n$, its tangent space can be characterized by
\begin{equation} \label{e35}
\T_x \mathcal{M} = \{\eta \in \mathbb{R}^n : B_x^T \eta = 0 \},
\end{equation}
where $B_x^T: \mathbb{R}^n \rightarrow \mathbb{R}^{n-d}: v \mapsto ( \inner[\F]{b_1}{v}, \inner[\F]{b_2}{v}, \ldots, \inner[\F]{b_{n-d}}{v} )^T$ is a linear operator, $d$ is the dimension of the manifold $\mathcal{M}$, and $\{b_1, b_2, \ldots, b_{n-d}\}$ forms an orthonormal basis of the normal space of $\T_x \mathcal{M}$.
{Concrete expressions of $B_x^T$  for various manifolds will be given later in Section~\ref{sec:Impl}.}
Based on $B_x^T$, Problem~\eqref{e34} can be written as
\begin{equation} \label{e09}
\tilde{\eta}_x = \argmin_{B_x^T \eta = 0} \tilde{\ell}_x(\eta) = \langle \grad f(x),\eta \rangle_{\F}+\frac{\tilde{L}}{2}\langle \eta, \eta\rangle_{\F} + g(x+\eta).
\end{equation}
Semi-smooth Newton method can be used to solve~\eqref{e09}. Specifically, the KKT condition of~\eqref{e09} is given by
\begin{align}
&\partial_{\eta} \mathcal{L}(\eta, \Lambda) = 0, \label{e36} \\
&B_x^T \eta = 0, \label{e37}
\end{align}
where $\mathcal{L}(\eta, \Lambda)$ is the Lagrangian function defined by
\begin{equation*}
\mathcal{L}(\eta,\Lambda)= \inner[\F]{\grad f(x)}{\eta} + \frac{\tilde{L}}{2} \inner[\F]{\eta}{\eta} + g(X+\eta) - \inner[\F]{\Lambda}{B_x^T \eta}.
\end{equation*}
Equation~\eqref{e36} yields
\begin{equation} \label{e38}
\eta=\kwrev{v(\Lambda)}:=\Prox_{g / \tilde{L}}\left(x - \frac{1}{\tilde{L}} (\grad f(x) - B_x \Lambda ) \right) - x,
\end{equation}
where 
\begin{equation}\label{e39}
\Prox_{g / \tilde{L}}(z) = \argmin_{v\in\mathbb{R}^{n}} \frac{1}{2}\|v-z\|^2+ \frac{1}{\tilde{L}} g(v)
\end{equation}
denotes the Euclidean proximal mapping. Substituting~\eqref{e38} into~\eqref{e37} yields that
\begin{equation} \label{e40}
\Psi(\Lambda):=B_x^T \left( \Prox_{g / \tilde{L}} \left(x - \frac{1}{\tilde{L}} (\grad f(x) - B_x \Lambda ) \right) - x \right)=0,
\end{equation}
which is a system of nonlinear equations with respect to $\Lambda$. Therefore, to solve~\eqref{e11}, one can first find the root of~\eqref{e40} and substitute it back to~\eqref{e38} to obtain $\tilde{\eta}_x$.
Moreover, the semi-smooth Newton method can be used to solve \eqref{e40}, which updates the iterate $\Lambda_k$ by
$\Lambda_{k+1} = \Lambda_k + d_k$,
where $d_k$ is a solution of
\[
J_{\Psi}(\Lambda_k) [d_k] = - \Psi(\Lambda_k),
\]
and $J_{\Psi}(\Lambda_k)$ is a generalized Jacobian of $\Psi$.

\whcomm{[I deleted the old Algorithm 4.1 and slightly modified the paragraph below.]}{}
\kwrev{To solve the proximal mapping \eqref{e34} approximately, we consider  an algorithm that can solve \eqref{e40} globally, e.g., the regularized semi-smooth Newton algorithm from~\cite{QS2006,ZST2010,XLWZ2016}. 
Given an approximate solution $\hat{\Lambda}$ to \eqref{e40}, define\footnote{Note that if $\Psi(\Lambda) \neq 0$, then $\eta$ defined by~\eqref{e38} may be not in $\T_x \mathcal{M}$. Therefore, we add an orthogonal projection.} 
\begin{equation} \label{e59}
\hat\eta_x = P_{\T_{x} \mathcal{M}} (v(\hat{\Lambda})),
\end{equation}
where $v(\cdot)$ is defined in \eqref{e38}.
We will show later, in order for $\hat{\eta}_x$ to satisfy~\eqref{e01}, it suffices to require   $\hat{\Lambda}$  to satisfy
\begin{align}
\|\Psi(\hat{\Lambda})\| &\leq \min( \phi(\| \hat{\eta}_x \|), 0.5),  \label{e58} \\
 \ell_x(0) &\geq \ell_x(\hat\eta_x), \label{e60}
\end{align}
where $\ell_x(\cdot)$ the Riemannian proximal mapping function used in Algorithm~\ref{a1}, $\phi:\mathbb{R} \rightarrow \mathbb{R}$ satisfies $\phi(0)=0$ and nondecreasing. Moreover, a globally convergent algorithm will terminate properly under these two stopping conditions.
} 
The analyses rely on Assumption~\ref{as02}.

\begin{assumption} \label{as02}
The function $g$ is convex and Lipschitz continuous with constant $L_g$, where the convexity and Lipschitz continuity are in the Euclidean sense.
\end{assumption}
Note that if $g$ is given by the one-norm regularization, then Assumption~\ref{as02} holds.

\whcomm{[Modified below $q$ to $\tilde{q}$.]}{}
\kwrev{It is evident that, in order to show that $\hat{\eta}_x$ satisfies \eqref{e01}, we only need to show that  there is a function $\tilde{q}(t)$ such that $\|\hat{\eta}_{x} - {\eta}_{x}^*\| \leq \tilde{q}(\|\hat{\eta}_{x}\|)$ holds if $\hat{\eta}_x$ satisfies \eqref{e58} and~\eqref{e60}. Therefore, the function $q(s, t)$ in~\eqref{e01} can be defined by $q(s, t) = \tilde{q}(t)$.}
\begin{theorem}\label{th02}
Suppose there exists a constant $\rho > 0$ such that for any $x \in \Omega_{x_0}$ it holds that $\Omega_{x_0} \subseteq R_x(\mathcal{B}(0_x, \rho))$.
If $\tilde{L}$ is sufficiently large and  the search direction $\hat{\eta}_x$ define in~\eqref{e59} satisfies \eqref{e58}, then we have
\begin{align}
&\|\hat{\eta}_{x} - {\eta}_{x}^*\| \leq \tilde{q}(\|\hat{\eta}_{x}\|), \label{e43} 
\end{align}
where 
\[
\tilde{q}(t) = \frac{2 L_g \varkappa_2 }{ \tilde{L} - 2 L_g \varkappa_2 } t + \sqrt{ \frac{ 4 L_g \varkappa_2 - 4 L_g^2 \varkappa_2^2 }{ (\tilde{L} - 2 L_g \varkappa)^2 } t^2 + \frac{4 \vartheta}{\tilde{L} - 2 L_g \varkappa_2} \min(\phi(t), 0.5) }
\]
and $\phi(t)$ is defined in~\eqref{e58}.
\end{theorem}

\begin{proof}
For ease of notation, let $\epsilon=\Psi(\hat{\Lambda}) = B_x^T v (\hat{\Lambda})$. 
Consider the optimization problem
\begin{align} \label{e24}
\min_{B_x^T \eta = \epsilon} \tilde{\ell}_x(\eta).
\end{align}
Its KKT condition is given by
\begin{equation*}
\partial_{\eta} \mathcal{L}(\eta, \Lambda) = 0, \qquad B_x^T \eta = \epsilon,
\end{equation*}
which is satisfied by \kwrev{$(\hat{\Lambda},v(\hat{\Lambda}))$} 
Therefore, $v(\hat{\Lambda})$ is the minimizer of $\tilde{\ell}_x(\eta)$ over the set $\mathcal{S} = \{v : B_x^T v = \epsilon \}$, i.e.,
\begin{equation} \label{e25}
v(\hat{\Lambda}) = \argmin_{v \in \mathcal{S}} \tilde{\ell}_x(\eta) = \langle \grad f(x), \eta \rangle_{\F}+\frac{\tilde{L}}{2}\langle \eta, \eta\rangle_{\F} + g(x+\eta).
\end{equation}
Define $\hat{\ell}_x(\eta_x) = \tilde{\ell}_x(\eta + B_x \epsilon)$. Further by the definition of $\hat{\eta}_x$, i.e., $\hat{\eta}_x = P_{\T_x \mathcal{M}} v(\hat{\Lambda})$, it is not hard to see that
\[
\hat{\eta}_x = \argmin_{\eta \in \T_x \mathcal{M}} \hat{\ell}_x(\eta).
\]

By the $\tilde{L}$-strongly convexity of $\hat{\ell}_x$ and the definition of $\hat{\eta}_x$, it holds that
\begin{equation} \label{e20}
\hat{\ell}_x(\eta_x) \geq \hat{\ell}_x( \hat{\eta}_x ) + \frac{\tilde{L}}{2} \|\eta_x - \hat{\eta}_x\|^2,\quad \forall \eta_x \in \T_x \mathcal{M}.
\end{equation}
By definition of $\hat{\eta}_x$, we have 
\begin{equation} \label{e32}
0 \in \whcomm{}{\grad f(x) + \tilde{L} \hat{\eta}_x + P_{\T_x \mathcal{M}} \partial^E g(x+\hat{\eta}_x + B_x \epsilon).}
\end{equation}
Since $\Omega_{x_0}$ is compact, there exists a constant $U_f$ such that $\|\grad f(x)\| < U_f$ for all $x \in \Omega_{x_0}$. By~\cite{CLA90}, if a function is Lipschitz continuous, then the norm of any subgradient is smaller than its Lipschitz constant. Therefore, by Assumption~\ref{as02}, it holds that $\|\zeta\| \leq L_g$ for any $\zeta \in \partial^E g(x + \eta)$. It follows from~\eqref{e32} and~\eqref{e58} 
that
\begin{equation} \label{e41}
\kwrev{\|\hat{\eta}_x\|} \leq \frac{U_f + L_g}{\tilde{L}}.
\end{equation}
\kwrev{Define $\mathcal{U} = \{R_x(\eta_x) : x \in \Omega_{x_0},  \|\eta_x\| \leq \rho \}$. Therefore, $\mathcal{U}$ is compact. Moreover, since $\Omega_{x_0} \subseteq R_x(\mathcal{B}(0_x, \rho))$  for any $x\in\Omega_{x_0}$, we have $\Omega_{x_0}\subset\mathcal{U}$.} It follows from~\eqref{e11} that there exists $\varkappa_2$ such that
\begin{equation} \label{e33}
\|R_x(\eta_x) - x - \eta_x\| \leq \varkappa_2 \|\eta_x\|^2
\end{equation}
holds for any $x \in \Omega_{x_0}$ and $\|\eta_x\| \leq \rho$. By Assumption~\ref{as02} and~\eqref{e33}, we have
\begin{equation} \label{e12}
|\ell_x(\eta_x) - \tilde{\ell}_x(\eta_x)| \leq L_g \varkappa_2 \|\eta_x\|^2, \quad \forall x \in \Omega_{x_0}, \quad \eta_x \in \T_x \mathcal{M}, \quad \|\eta_x\| \leq \rho.
\end{equation}
Moreover, by the definition of $\hat{\ell}_x(\eta_x)$, we have for any \kwrev{$x\in\Omega_{x_0}$ and $\|\eta_x\|\leq \rho$},
\begin{align*}
    &|\hat{\ell}_x(\eta_x)-\tilde{\ell}_x(\eta_x)|\\
    &\leq \|\grad f(x)\|\|B_x\epsilon\|+\tilde{L}\|\eta_x\|\|B_x\epsilon\|+\frac{\tilde{L}}{2}\|B_x\epsilon\|^2+|g(x+\eta_x+B_x\epsilon)-g(x+\eta_x)|\\
    &=(\|\grad f(x)\|+\tilde{L}\|\eta_x\|+\frac{\tilde{L}}{2}\|B_x\epsilon\|+L_g)\|B_x\epsilon\|\\
    &\leq (U_f+\kwrev{(\rho+1)}\tilde{L}+L_g)\|B_x\epsilon\|\\
    &=:\vartheta_2\|B_x\epsilon\|,
\end{align*}
where the third line has used the fact $\|B_x\epsilon\|\leq \|\epsilon\|\leq 1/2$ (see \eqref{e58}).
Together with ~\eqref{e12}, and~\eqref{e20}, it holds that for any $x \in \Omega_{x_0}, \eta_x \in \T_x \mathcal{M}$, and $\|\eta_x\| \leq \rho$,
\begin{align}
&\hat{\ell}_x( \hat{\eta}_x ) + \frac{\tilde{L}}{2} \|\eta_x - \hat{\eta}_x\|^2 - L_g \varkappa_2 \|\eta_x\|^2 - \vartheta_2 \|B_x \epsilon\| \nonumber \\
\leq& \hat{\ell}_x(\eta_x) - L_g \varkappa_2 \|\eta_x\|^2 - \vartheta_2 \|B_x \epsilon\| \nonumber \\
\leq& \tilde{\ell}_x(\eta_x) - L_g \varkappa_2 \|\eta_x\|^2\\
\leq& \ell_x(\eta_x) \\
\leq& \tilde{\ell}_x(\eta_x) + L_g \varkappa_2 \|\eta_x\|^2 \nonumber \\
\leq& \hat{\ell}_x(\eta_x) + L_g \varkappa_2 \|\eta_x\|^2 + \vartheta_2 \|B_x \epsilon\|. \label{e42}
\end{align}

Define $$\hat{\Omega} = \{ \eta_x \in \T_x \mathcal{M} : \frac{\tilde{L}}{2} \|\eta_x - \hat{\eta}_x\|^2 - L_g \varkappa_2 \|\eta_x\|^2 - \vartheta_2 \|B_x \epsilon\| \leq L_g \varkappa_2 \|\hat{\eta}_x\|^2 + \vartheta_2 \|B_x \epsilon\|
 \}.$$ It is easy to verify that
\begin{align*}
 \hat{\Omega} = &\left\{ \eta_x \in \T_x \mathcal{M} : \left\|\eta_{x} - \frac{\tilde{L}}{\tilde{L} - 2 L_g \varkappa_2} \hat{\eta}_x\right\|{\color{white}\sqrt{\frac{ 4 \tilde{L} L_g \varkappa_2 - 4 L_g^2 \varkappa_2^2}{ (\tilde{L} - 2 L_g \varkappa)^2 }}} \right. \\
 & \qquad \qquad \leq \left. \sqrt{ \frac{ 4 \tilde{L} L_g \varkappa_2 - 4 L_g^2 \varkappa_2^2}{ (\tilde{L} - 2 L_g \varkappa_2)^2 } \|\hat{\eta}_x\|^2 + \frac{4 \vartheta_2}{\tilde{L} - 2 L_g \varkappa_2} \|B_x \epsilon\|} \right\},
\end{align*}
which yields
\begin{align*}
\hat{\Omega} \subseteq \mathcal{U} :=& \left\{ \eta_x \in \T_x \mathcal{M} : \| \eta_{x} - \hat{\eta}_x \| \leq{\color{white}\sqrt{\frac{ 4 \tilde{L} L_g \varkappa_2 - 4 L_g^2 \varkappa_2^2}{ (\tilde{L} - 2 L_g \varkappa)^2 }}} \right. \\
&\left. \frac{2 L_g \varkappa_2 }{ \tilde{L} - 2 L_g \varkappa_2 } \|\hat{\eta}_x\| + \sqrt{ \frac{ 4 \tilde{L} L_g \varkappa_2 - 4 L_g^2 \varkappa_2^2}{ (\tilde{L} - 2 L_g \varkappa_2)^2 } \|\hat{\eta}_x\|^2 + \frac{4 \vartheta_2}{\tilde{L} - 2 L_g \varkappa_2} \|B_x \epsilon\|} \right\}.
\end{align*}
\kwrev{Noting the expression of $\vartheta_2$, when $\tilde{L}\rightarrow\infty$, the righthand side in the above inequality goes to $\sqrt{4(\rho+1)}$. Thus, for sufficiently large $\tilde{L}$ and $\rho$, we have 
\begin{align*}
    \hat{\Omega}&\subset\mathcal{U} := \{\eta_x \in \T_x \mathcal{M} : \| \eta_{x} \| \leq \rho/2 \}\\
    &\subset\mathcal{W} := \{\eta_x \in \T_x \mathcal{M} : \| \eta_{x} \| \leq \rho \}.
\end{align*}}

For any $\eta_x \in \mathcal{W}$ but not in $\hat{\Omega}$, it follows from~\eqref{e42} that
\begin{align}
\ell_x(\eta_x) \geq& \hat{\ell}_x( \hat{\eta}_x ) + \frac{\tilde{L}}{2} \|\eta_x - \hat{\eta}_x\|^2 - L_g \varkappa_2 \|\eta_x\|^2 - \vartheta \|B_x \epsilon\| \nonumber \\
>& \hat{\ell}_x( \hat{\eta}_x ) + L_g \varkappa_2 \|\hat{\eta}_x\|^2 + \vartheta \|B_x \epsilon\| \geq \ell_x(\hat{\eta}_x).  \label{e45}
\end{align}
Therefore, there exists a global minimizer of $\ell_x$ in the set $\hat{\Omega}$, we denote it by $\eta_x^*$. \kwrev{It follows from $\eta_x^* \in \hat{\Omega}$, and thus ${\|\hat{\eta}_{x} - {\eta}_{x}^*\| \leq \tilde{q}(\|\hat{\eta}_{x}\|)}$ for $\tilde{q}(t)$ given in the theorem.}
\end{proof}

Theorem~\ref{th02} ensures that the search direction  given by~\eqref{e59} 
is desirable for IRPG to have global convergence. There are several implications of this theorem. First, the global convergence of ManPG in~\cite{CMSZ2019} follows and the step size one is always acceptable. \whcomm{}{This can be seen by noting that if $\phi(t) \equiv 0$, then the direction
$\hat\eta$ 
with $\Lambda$ satisfying~\eqref{e58} is the search direction used in~\cite{CMSZ2019}.}
Secondly, one can relax the accuracy of the solution in ManPG and still guarantees its global convergence. 
However, it should be pointed out that Theorem~\ref{th02} does not implies that $\hat{\eta}_x$ satisfies~\eqref{e28} or~\eqref{e29}. Therefore, the uniqueness of accumulation points and the convergence rate based on the Riemannian KL property are not guaranteed. 

{Lemma~\ref{le04} shows that \kwrev{a globally convergent algorithm for solving \eqref{e40}} can terminate properly in the sense that it satisfies~\eqref{e58} and~\eqref{e60} under the assumption that $\tilde{L}$ is sufficiently large.}
\begin{lemma} \label{le04}
Suppose there exists a constant $\rho > 0$ such that for any $x \in \Omega_{x_0}$ it holds that $\Omega_{x_0} \subseteq R_x(\mathcal{B}(0_x, \rho))$.
If $\tilde{L}$ is sufficiently large and \kwrev{an algorithm that converges globally is used for \eqref{e40}}, then there exists an iterate from the algorithm such that $\hat{\eta}_x$ satisfies~\eqref{e58} and~\eqref{e60}.
\end{lemma}

\begin{proof}

If $\epsilon = 0$, then $\hat{\eta}_x = \tilde{\eta}_x$ and the above derivations for $\hat{\eta}_x$ also hold for $\tilde{\eta}_x$. Therefore, $\ell_x(0_x) > \ell_x(\tilde{\eta}_x)$ follows from~\eqref{e45} by noting $0_x \in \mathcal{W}$ and $0_x \notin \hat{\Omega}$ \kwrev{when $\tilde{L}$ is sufficiently large.}.
Finally, by strong convexity of $\tilde{\ell}_x$ and the convergence of the  algorithm, we have that $\hat{\eta}_x \rightarrow \tilde{\eta}_x$ and $\|\Psi(\Lambda)\| \rightarrow 0$. 
Therefore, the iterate $\hat{\eta}_x$ satisfies~\eqref{e58} and~\eqref{e60} in certain iteration.
\end{proof}

\subsection{\kwrev{Practical Condition that Ensures Local Convergence Rate}}  \label{sec:LocalRate}

\kwrev{In this section, we  directly consider the solution of the Riemanniann proximal mapping \eqref{e61}  and provide a   practical condition that meets the requirement for the local convergence rate analysis.} First note that the Riemnnaian proximal mapping  \eqref{e61} is equivalent to
\begin{equation} \label{e62}
\min_{c \in \mathbb{R}^d} J_x(c) := \inner[\F]{c}{Q_x^T \grad f(x)} + \frac{\tilde{L}}{2} \|c\|^2 + g(R_{x}(Q_{x} c)),
\end{equation}
which is an optimization problem on a Euclidean space, where the subscript $k$ is omitted for simplicity, $d$ is the dimension of $\mathcal{M}$ and $Q_x$ forms an orthonormal space of $\T_x \mathcal{M}$.

The analysis in this section relies on the notion of error bound (see its definition in e.g.,~\cite[(35)]{TY2009}, \cite{ZS2017}), whose discussion relies on the convexity of the objective function. Therefore, we will make Assumption~\ref{as05} which uses Definition~\ref{def:Rconvexity}. It follows that $J_x(c)$ is convex. Note that Definition~\ref{def:Rconvexity} has also been used in~\cite{HGA2014,HuaWei2019b} 
\begin{definition} \label{def:Rconvexity}
A function $h:\mathcal{M} \rightarrow \mathbb{R}$ is called retraction-convex with respect to a retraction $R$ in  $\mathcal{N} \subseteq \mathcal{M}$ if for any $x \in \mathcal{N}$ and any $\mathcal{S}_x \subseteq \T_x \mathcal{M}$ such that $R_x(\mathcal{S}_x)\subseteq \mathcal{N}$, there exists a tangent vector $\zeta \in \T_x \mathcal{M}$ such that $p_x = h \circ R_x$ satisfies
\begin{equation} \label{RPG:def:conv}
p_x(\eta) \geq p_x(\xi) + \inner[x]{\zeta}{\eta - \xi}\;\; \forall \eta, \xi \in \mathcal{S}_x.
\end{equation}
Note that $\zeta = \grad p_x(\xi)$ if $h$ is differentiable; otherwise, $\zeta$ is any Riemannian subgradient of $p_x$ at $\xi$.
\end{definition}
\begin{assumption} \label{as05}
The function $g$ is retraction-convex on $\mathcal{M}$.
\end{assumption}

\kwrev{In the typical error bound analysis, the residual map plays a key role which controls the distance of a point to the optimal solution set. For our purpose, the residual map  for~\eqref{e62} is defined as follows:}
\begin{equation} \label{e67}
r_x(c) = \argmin_{v \in \mathbb{R}^n} w_{x, c}(v) := \inner[\F]{v}{Q_x^T \grad f(x) + \tilde{L} c} + \frac{\tilde{L}}{2} \|v\|^2 + g(R_x( Q_x(c + v))),
\end{equation}
\kwrev{It is not hard to see that 
\begin{align*}
    r_x(c^*)=0\Leftrightarrow c^*\mbox{ is the optimal solution to \eqref{e62}}.
\end{align*}Note the residual map defined here is slightly different from the one defined in \cite{TY2009}, where the coefficient in front $\|v\|^2$ is $1/2$ instead of $\tilde{L}/2$. However, the error bound can be established in  exactly the same way. To keep the presentation self-contained, details of the proof are provided below. }
It is worth pointing out that the family of Problems~\eqref{e62} parameterized by $x$ possesses an error bound property with the coefficient independent of $x$.
\begin{lemma} \label{le05}
Suppose that Assumption~\ref{as05} holds. Then it holds that
\begin{equation} \label{e83}
\|c - c_x^*\| \leq 2 \|r_{x}(c)\|, \quad \hbox{ for all } x \in \mathcal{M},
\end{equation}
where $c_x^*$ is the minimizer of $J_x(c)$.

\end{lemma}

\begin{proof}

Let $\tilde{f}_x(c)$ denote $\grad f(x)^T Q_{x} c + \frac{\tilde{L}}{2} \|c\|^2$ and $\tilde{g}_{x}(c)$ denote $g(R_{x}(Q_{x}(c)))$. It follows that $J_x(c) = \tilde{f}_x(c) + \tilde{g}_x(c)$ and
\[
r_{x}(c) = \argmin_{v \in \mathbb{R}^d} \inner[\F]{v}{\nabla \tilde{f}_x(c)} + \frac{\tilde{L}}{2} \|v\|^2 + \tilde{g}_x(c + v).
\]
Therefore, we have
$
0 \in \nabla \tilde{f}_x(c) + \tilde{L} r_{x}(c) + \partial^E \tilde{g}_x(c + r_{x}(c)),
$
which implies
\[
r_{x}(c) = \argmin_{v \in \mathbb{R}^d} \inner[\F]{\nabla \tilde{f}_x(c) + \tilde{L} r_{x}(c)}{v} + \tilde{g}_x(c + v).
\]
It follows that
\begin{equation} \label{e74}
\inner[\F]{\nabla \tilde{f}_x(c) + \tilde{L} r_{x}(c)}{r_{x}(c)} + \tilde{g}_x(c + r_{x}(c)) \leq \inner[\F]{\nabla \tilde{f}_x(c) + \tilde{L} r_{x}(c)}{c_x^* - c} + \tilde{g}_x(c_x^*).
\end{equation}
Since $0 \in \nabla \tilde{f}_x(c_x^*) + \partial^E \tilde{g}_x(c_x^*)$, we have
$
c_x^* = \argmin_{v \in \mathbb{R}^n} \nabla \tilde{f}_x(c_x^*)^T v + \tilde{g}_x(v).
$
Therefore,
\begin{equation} \label{e75}
\inner[\F]{\tilde{f}_x(c_x^*)}{c_x^*} + \tilde{g}_x(c_x^*) \leq \inner[\F]{\nabla \tilde{f}_x(c_x^*)}{c + r_{x}(c)} + \tilde{g}_x(c + r_{x}(c)).
\end{equation}
Adding~\eqref{e74} to~\eqref{e75} yields
\begin{equation}\label{e76}
\inner[\F]{\tilde{f}_x(c) - \tilde{f}_x(c_x^*)}{c - c_x^*} + \tilde{L}\|r_{x}(c)\|^2 \leq \inner[\F]{\tilde{f}_x(c_x^*) - \tilde{f}_x(c)}{r_{x}(c)} + \tilde{L} \inner[\F]{r_{x}(c)}{c_x^* - c}.
\end{equation}
By definition of $\tilde{f}_x$, we have that $\tilde{f}_x$ is $\tilde{L}$-strongly convex and Lipschitz continuously differentiable with constant $\tilde{L}$. Therefore, \eqref{e76} yields 
\[
\tilde{L} \|c - c_x^*\|^2 \leq 2 \tilde{L} \|c - c_x^*\| \|r_{x}(c)\|,
\]
which implies $ \|c - c_x^*\| \leq 2 \|r_{x}(c)\|.$
\end{proof}

Computing the residual map~\eqref{e67} is usually impractical due to the existence of the retraction $R$ in $g$. Therefore, we use the same technique in~\cite[Section~3.5]{HuaWei2019b} to linearize $R_x(Q_x(c + v))$ by $R_x(Q_x c) + \mathcal{T}_{R_{Q_{x} c}} Q_{x} v$, \kwrev{and define a new residual map $\tilde{r}_{x}(c)$ that can be used to upper bound $r_x(c)$,
\begin{equation} \label{e68}
\tilde{r}_{x}(c) = \argmin_{v \in \mathbb{R}^d} \tilde{w}_{x, c}(v) := \inner[\F]{v}{ \grad f(x) + \tilde{L} Q_x c } + \frac{\tilde{L}}{2} \|\mathcal{T}_{R_{Q_{x} c}} Q_{x} v \|^2 + g(y + \mathcal{T}_{R_{Q_{x} c}} Q_{x} v ),
\end{equation}
where $y = R_x(Q_x c)$. A simple calcualtion can still show that 
\begin{align*}
    \tilde{r}_x(c^*)=0\Leftrightarrow c^*\mbox{ is the optimal solution to \eqref{e62}}.
\end{align*}
}Moreover, minimizing $\tilde{w}$ is the same as Problem~\eqref{e34} and therefore can be solved by the techniques in Section~\ref{sect:subprobGlobal}.
\begin{lemma} \label{le06}
 Let $\mathcal{G} \subset \mathcal{M}$ be a compact set. 
Suppose that Assumptions~\ref{as02} and~\ref{as05} hold, and that there exists a parallelizable set $\mathcal{U}$ such that $\mathcal{G} \subset \mathcal{U}$, where a set is callel parallelizable if $Q_x$ as a function of $x$ is smooth in $\mathcal{U}$\footnote{The notion of a parallelizable set is defined in~\cite{HAG13} and the function $Q$ is also called a local frame. The existence of a smooth $Q$ around any point $x \in \mathcal{M}$ can be found in~\cite{AMS2008,boumal2020intromanifolds}.}.
If $\tilde{L}$ is sufficient large, then there exist two constants $b>0$ and $\delta > 0$ such that
\begin{equation} \label{e71}
\|r_{x}(c)\| \leq b \|\tilde{r}_{x}(c)\|
\end{equation}
for all $x \in \mathcal{G}$ and $\|c\| < \delta$. 
\end{lemma}
\begin{proof}
Since $Q_x$ is smooth in $\mathcal{U}$ and $\mathcal{T}_R$ is smooth, we have that the function $z:\mathcal{U} \times \mathbb{R}^d \rightarrow \mathbb{R}^{d \times d}: (x, c) \mapsto Q_y^T \mathcal{T}_{R_{Q_x c}} Q_x$ is a smooth function, where $y = R_x(Q_x c)$. Furthermore, since $\mathcal{T}_{R_{0_x}}$ is an identity for any $x \in \mathcal{M}$, we have $z(x, 0) = I_d$ for any $x \in \mathcal{M}$. It follows that
\begin{equation} \label{e77}
\|Q_y^T \mathcal{T}_{R_{Q_x c}} Q_x - I_d\| \leq L_J \|c\|, \hbox{ for any } x \in \mathcal{G}, \|c\| \leq \delta,
\end{equation}
where $L_J = \max_{x \in \mathcal{G}, \|c\|\leq \delta} \|J_z(x, Q_x c)\|$. Since the set of $\{(x, c) : x \in \mathcal{G}, \|c\|\leq \delta\}$ is compact and the Jacobi $J_z$ is continuous by smoothness of $z$, we have $L_J < \infty$.

Using~\eqref{e77} and noting $\|\mathcal{T}_{R_{Q_x c}}\| = \|Q_y^T \mathcal{T}_{R_{Q_x c}} Q_x\|$ yields
\begin{align*}
\|\mathcal{T}_{R_{Q_x c}}\| \leq \|I_d\| + \|Q_y^T \mathcal{T}_{R_{Q_x c}} Q_x - I_d\| \leq 1 + L_J \|c\|
\end{align*}
and
\begin{align*}
\|\mathcal{T}_{R_{Q_x c}}^{-1}\| - \|I_d\| &\leq \|I_d - (Q_y^T \mathcal{T}_{R_{Q_x c}} Q_x)^{-1}\| \\
&\leq \|\mathcal{T}_{R_{Q_x c}}^{-1}\| \| Q_y^T \mathcal{T}_{R_{Q_x c}} Q_x - I_d\| \leq L_J \|c\| \|\mathcal{T}_{R_{Q_x c}}^{-1}\|,
\end{align*}
which gives
\begin{equation*}
(1 - L_J \|c\|) \|\mathcal{T}_{R_{Q_x c}}^{-1}\| \leq 1.
\end{equation*}
Therefore, by choosing $\delta < \min(\sqrt{3/2}-1, 1 - 1/\sqrt{2}) / L_J$, we have
\begin{equation} \label{e78}
\|\mathcal{T}_{R_{Q_x c}}\| \leq \sqrt{3 / 2}, \hbox{ and } \|\mathcal{T}_{R_{Q_x c}}^{-1}\| \leq \sqrt{2}.
\end{equation}
\whcomm{}{
It follows that
\begin{align*}
\|\mathcal{T}_{R_{Q_x c}} Q_x v\|^2 - \|v\|^2 \leq (\|\mathcal{T}_{R_{Q_x c}}\|^2 - 1) \|v\|^2 \leq \frac{1}{2} \|v\|^2 \hbox{ and } \\
\|\mathcal{T}_{R_{Q_x c}} Q_x v\|^2 - \|v\|^2 \geq \left( \frac{1}{\|\mathcal{T}_{R_{Q_x c}}^{-1}\|^2} - 1 \right) \|v\|^2 \geq - \frac{1}{2} \|v\|^2,
\end{align*}
which yields
\[
|\|\mathcal{T}_{R_{Q_x c}} Q_x v\|^2 - \|v\|^2| \leq \frac{1}{2} \|v\|^2.
\]
}
By the compactness of $\mathcal{G}$, there exists a constant $\tilde{\chi}_2$ such that
\begin{equation} \label{e69}
\|R_x(Q_x (c + v)) - R_x(Q_x c) - \mathcal{T}_{R_{Q_x c}} Q_x v \| \leq \tilde{\chi}_2 \|\mathcal{T}_{R_{Q_x c}} Q_x v\|^2,
\end{equation}
for all $x, R_x(Q_x c), R_x(Q_x (c+v)) \in \mathcal{G}$.

Therefore, we have
\begin{align*}
|w_{x, c}(v) &- \tilde{w}_{x, c}(v)| \\
\leq& \left|\frac{\tilde{L}}{2} \|v\|^2 + g(R_x( Q_x(c + v))) - \frac{\tilde{L}}{2} \|\mathcal{T}_{R_{Q_{x} c}} Q_{x} v \|^2 + g(y + \mathcal{T}_{R_{Q_{x} c}} Q_{x} v )\right| \\
\leq& \frac{\tilde{L}}{2} |\|\mathcal{T}_{R_{Q_x c}} Q_x v\|^2 - \|v\|^2| + L_g \tilde{\chi}_2 \|\mathcal{T}_{R_{Q_x c}}\|^2 \|v\|^2 \\
\leq& C_R \|v\|^2,
\end{align*}
where $C_R = \tilde{L} / 4 + 2 L_g \tilde{\chi}_2$, the second inequality follows from~\eqref{e69} and Assumption~\ref{as02}, and the last inequality follows from~\eqref{e78}.

Since $w_{x, c}$ and $\tilde{w}_{x, c}$ are both strongly convex, their minimizers $r_{x}(c) \in \mathbb{R}^d$ and $\tilde{r}_{x}(c) \in \T_y \mathcal{M}$ are unique. 
By the same derivation in Theorem~\ref{th02}, 
we have that
\[
\frac{\tilde{L}}{2} \|r_{x}(c) - \tilde{r}_{x}(c)\|^2 - C_R \| r_{x}(c) \|^2 \leq C_R \| \tilde{r}_{x}(c) \|^2,
\]
which implies
\begin{align*}
\left( \sqrt{\frac{\tilde{L}}{2}} - \sqrt{C_R} \right) \|r_{x}(c)\| \leq& \left( \sqrt{\frac{\tilde{L}}{2}} + \sqrt{C_R} \right) \| \tilde{r}_{x}(c) \|, \quad \hbox{ for all $k > k_0$.}
\end{align*}
By assuming $\tilde{L} > 8 L_g \tilde{\chi}_2$, we have that~\eqref{e71} holds with $b = \frac{\sqrt{\tilde{L}} + \sqrt{2 C_R}}{\sqrt{\tilde{L}}  - \sqrt{2 C_R}}$.
\end{proof}

The main result is stated in Theorem~\ref{th04}, which follows from Lemmas~\ref{le05} and~\ref{le06}. 
It shows that if the Riemannian proximal mapping is solved sufficiently accurate such that the computable $\tilde{r}_{x_k}(\bar{c}_k)$ satisfies~\eqref{e72}, then the difference $\|\bar{\eta}_{x_k} - \eta_{x_k}^*\|$ is controlled from above by the prescribed function $\psi$.
An algorithm that  achieves \eqref{e72} can be found in~\cite[Algorithm~2]{HuaWei2019b} by adjusting its stopping criterion to~\eqref{e72}.
\begin{theorem} \label{th04}
Let $\mathcal{S}$ denote the set of all accumulation points of $\{x_k\}$. Suppose that there exists a neighborhood of $\mathcal{S}$, denoted by $\mathcal{U}$, such that $\mathcal{U}$ is a parallelizable set, that Assumptions~\ref{as10}, ~\ref{as3}, ~\ref{as02} and~\ref{as05} hold, that $\tilde{L}$ is sufficiently large, and
that an algorithm is used to solve
\begin{equation*}
\min_{c \in \mathbb{R}^d} J_{x_k}(c) := \inner[\F]{c}{Q_x^T \grad f(x_k)} + \frac{\tilde{L}}{2} \|c\|^2 + g(R_{x_k}(Q_{x_k} c)),
\end{equation*}
such that the output $\bar{c}_k \in \mathbb{R}^d$ of the algorithm satisfies
\begin{equation} \label{e72} 
\|\tilde{r}_{x_k}(\bar{c}_k)\| \leq \psi(\varepsilon_k, \varrho, \|\bar{c}_k\|),
\end{equation}
where $x_k$ is the $k$-th iterate of Algorithm~\ref{a1} and $\psi$ is a function from $\mathbb{R}^3$ to $\mathbb{R}$.
Then there exists a constant $\tilde{a} > 0$ and an integer $\tilde{K} > 0$ such that for all $k > \tilde{K}$, it holds that
\begin{equation} \label{e73} 
\|\bar{\eta}_{x_k} - \eta_{x_k}^*\| \leq \tilde{a} \psi(\varepsilon_k, \varrho, \|\bar{\eta}_{x_k}\|),
\end{equation}
where $\bar{\eta}_{x_k} = Q_{x_k} \bar{c}_k$.
Moreover, if $\psi(\varepsilon_k, \varrho, t) = \varepsilon_k^2$, then inequality~\eqref{e28} holds; if $\psi(\varepsilon_k, \varrho, t) = \min(\varepsilon_k^2, \varrho t^2)$ with $\varrho < \frac{\beta}{2 L_F a}$, then inequality~\eqref{e29} holds.
\end{theorem}
\begin{proof}
By~\eqref{eq:kwrev01a} and~\cite[Remark~5]{BST2014}, we have that $\mathcal{S}$ is a compact set. Therefore, there exists a compact set $\mathcal{G}$ and an integer $\tilde{K} > 0$ such that $\mathcal{S} \subset \mathcal{G} \subset \mathcal{U}$ and it holds that $x_k \subset \mathcal{G}$ for all $k > \tilde{K}$. By Lemma~\ref{le06}, there exists two constants $b> 0$ and $\delta > 0$ such that 
$
\|r_{x}(c)\| \leq b \|\tilde{r}_{x}(c)\|
$
for all $x \in \mathcal{G}$ and $\|c\| < \delta$. In addition, it follows from~\eqref{e79} that there exists a constant $\tilde{K}_+ > 0$ such that for all $k > \tilde{K}_+$ it holds that $\|\hat{\eta}_{x_k}\| < \delta$. Therefore, for all $k > \max(\tilde{K}, \tilde{K}_+)$, we have
\begin{equation} \label{e80}
\|r_{x_k}(\bar{c}_k)\| \leq b \|\tilde{r}_{x_k}(\bar{c}_k)\|.
\end{equation}
The result~\eqref{e73} follows from~\eqref{e72} and~\eqref{e80}.
\end{proof}

\whcomm{}{
For simplicity, we define $\tilde{r}_{x_k}(c)$ as the minimizer of $\tilde{w}_{x_k, c}(v)$. Indeed, we can show that it is not necessary to optimize $\tilde{w}_{x_k, c}(v)$ exactly. Suppose that the minimizer $c_{x_k}^*$ of $J_{x_k}(c)$ is nonzero, that a converging algorithm is used to optimize $J_{x_k}(c)$ and let $\{c_i\}$ denote the generated sequence, and that $\tilde{w}_{x_k, c}(v)$ is only solved approximately such that the approximated solution, denoted by $\tilde{\tilde{r}}_{x_k}(c_i)$, satisfies $\|\tilde{\tilde{r}}_{x_k}(c_i) - \tilde{r}_{x_k}(c_i)\| \leq \delta_r \|\tilde{\tilde{r}}_{x_k}(c_i)\|$, where $\delta_r \in (0, 1)$ is a constant
\footnote{Note that $\tilde{w}_{x_k, c}(v)$ has the same format as~\eqref{e34}. We can use condition~\eqref{e58} and~\eqref{e60} to find the approximate solution $\tilde{\tilde{r}}_{x_k}(c_i)$.}.
Then we have
\begin{equation} \label{e81}
(1 - \delta_r) \|\tilde{\tilde{r}}_{x_k}(c_i)\| \leq \|\tilde{r}_{x_k}(c_i)\| \leq (1 + \delta_r) \|\tilde{\tilde{r}}_{x_k}(c_i)\|.
\end{equation}
It follows that if 
\begin{equation} \label{e84}
\|\tilde{\tilde{r}}_{x_k}(c_i)\| \leq \frac{1}{1+\delta_r} \psi(\varepsilon_k, \varrho, \|c_i\|),
\end{equation}
then~\eqref{e72} holds. Since a converging algorithm is used, we have $c_i$ goes to $c_{x_k}^*$ and $\tilde{r}_{x_k}(c_i)$ goes to zero. It follows that $\psi(\varepsilon_k, \varrho, \|c_i\|)$ is greater than a positive constant for all $i$ and $\tilde{\tilde{r}}_{x_k}(c_i)$ goes to zero by~\eqref{e81}. Therefore, an iterate $c_i$, denoted by $\bar{c}_k$, satisfying inequality~\eqref{e84} can be found.
}

\slversions{}{
\subsection{Implementations of $B_x^T$ and $B_x$} \label{sec:Impl}

In this section, the implementations of the functions $B_x^T: \mathbb{R}^n \rightarrow \mathbb{R}^{n - d}$ and $B_x: \mathbb{R}^{n - d} \rightarrow \mathbb{R}^n$ are given for Grassmann manifold, manifold of fixed-rank matrices, manifold of symmetric positive definite matrices, and products of manifolds. Note that the Riemannian metric is chosen to be the Euclidean metric in this section.

\paragraph{Grassmann manifold:} We consider the representation of Grassmann manifold by
\[
\Gr(p, n) = \{[X] : X \in \St(p, n)\},
\]
where $[X] = \{X O : O^T O = I_p\}$. The ambient space of $\Gr(p, n)$ is $\mathbb{R}^{n \times p}$ and the orthogonal complement space of the horizontal space $\mathcal{H}_X$ at $X \in \St(p, n)$ is given by
\[
\mathcal{H}_X^{\perp} = \{X M : M \in \mathbb{R}^{p \times p} \}.
\]
Therefore, we have
\begin{align*}
&B_X^T: \mathbb{R}^{n \times p} \rightarrow \mathbb{R}^{p \times p}: Z \rightarrow X^T Z, \hbox{ and } \\
&B_X: \mathbb{R}^{p \times p} \rightarrow \mathbb{R}^{n \times p}: M \rightarrow X M.	
\end{align*}

\paragraph{Manifold of fixed-rank matrices:} The manifold is given by
\[
\mathbb{R}_r^{m \times n} = \{X \in \mathbb{R}^{m \times n} : \mathrm{rank}(X) = r\}.
\]
The ambient space is therefore $\mathbb{R}^{m \times n}$. Given $X \in \mathbb{R}_r^{m \times n}$, let $X = U S V$ be a thin singular value decomposition. The normal space at $X$ is given by
\[
\N_X \mathbb{R}_r^{m \times n} = \{U_{\perp} M V_{\perp}^T : M \in \mathbb{R}^{(m - r) \times (n - r)} \},
\]
where $U_{\perp} \in \mathbb{R}^{m \times (m - r)}$ forms an orthonormal basis of the perpendicular space of $\mathrm{span}(U)$ and $V_{\perp} \in \mathbb{R}^{n \times (n - r)}$ forms an orthonormal basis of the perpendicular space of $\mathrm{span}(V)$. It follows that
\begin{align*}
&B_X^T: \mathbb{R}^{m \times n} \rightarrow \mathbb{R}^{(m - r) \times (n - r)}: Z \mapsto U_{\perp}^T Z V_{\perp}, \hbox{ and }	\\
&B_X: \mathbb{R}^{(m - r) \times (n - r)} \rightarrow \mathbb{R}^{m \times n}: M \mapsto U_{\perp} M V_{\perp}^T.
\end{align*}
Note that it is not necessary to form the matrices $U_{\perp}$ and $V_{\perp}$. One can use~\cite[Algorithms~4 and~5]{HAG2016VT} to implement the actions of $U_{\perp}$, $U_{\perp}^T$, $V_{\perp}$, and $V_{\perp}^T$.

\paragraph{Manifold of symmetric positive semi-definite matrices:} The manifold is
\[
\mathbb{S}_r^{n \times n} = \{X \in \mathbb{R}^{n \times n} : X = X^T, X \succeq 0, \mathrm{rank}(X) = r \}.
\]
The ambient space is $\mathbb{R}^{n \times n}$. Given $X \in \mathbb{S}_r^{n \times n}$, let $X = H H^T$, where $H \in \mathbb{R}^{n \times r}$ is full rank. The normal space at $X$ is
\[
\N_X \mathbb{S}_r^{n \times n} = \{H_{\perp} M H_{\perp}^T: M \in \mathbb{R}^{(n - r) \times (n - r)}, M = M^T \},
\]
where $H_{\perp} \in \mathbb{R}^{n \times (n - r)}$ forms an orthonormal basis of the perpendicular space of $\mathrm{span}(H)$.
Therefore, we have
\begin{align*}
&B_X^T: \mathbb{R}^{n \times n} \rightarrow \mathbb{R}^{\frac{(n - r) (n - r + 1)}{2}}: Z \mapsto  \mathrm{vec}\left(\frac{1}{2} H_{\perp}^T (Z + Z^T) H_{\perp}\right), \hbox{ and }	\\
&B_X: \mathbb{R}^{\frac{(n - r) (n - r + 1)}{2}} \rightarrow \mathbb{R}^{n \times n}: v \mapsto H_{\perp} \mathrm{vec}^{-1}(v) H_{\perp}^T,
\end{align*}
where $\mathrm{vec}(M) = (M_{11}, M_{22}, \ldots, M_{ss}, \sqrt{2} M_{12}, \sqrt{2} M_{13}, \sqrt{2} M_{1s}, \ldots, \sqrt{2} M_{(s - 1) s})^T$ for $M \in \mathbb{R}^{s \times s}$ being a symmetric matrix, and $\mathrm{vec}^{-1}$ is the inverse function of $\mathrm{vec}$.

\paragraph{Product of manifolds:} Let the product manifold $\mathcal{M}$ be denoted by $\mathcal{M}_1 \times \mathcal{M}_2 \times \ldots \times \mathcal{M}_t$. Let the ambient space of $\mathcal{M}_i$ be $\mathbb{R}^{n_i}$ and the dimension of $\mathcal{M}_i$ be $d_i$. For any $X = (X_1, X_2, \ldots, X_t) \in \mathcal{M}$, the mappings $B_X^T$ and $B_X$ are given by
\begin{align*}
B_X^T&: \mathbb{R}^{n_1} \times \mathbb{R}^{n_2} \times \ldots \times \mathbb{R}^{n_t} \rightarrow \mathbb{R}^{(n_1 - d_1 + n_2 - d_2 + \ldots + n_t - d_t)} \\
&: (Z_1, Z_2, \ldots, Z_t) \mapsto \left( (B_{X_1}^T Z_1)^T, (B_{X_2}^T Z_2)^T, \ldots, (B_{X_t}^T Z_t)^T \right)^T, \hbox{ and } \\
B_X&: \mathbb{R}^{(n_1 - d_1 + n_2 - d_2 + \ldots + n_t - d_t)} \rightarrow \mathbb{R}^{n_1} \times \mathbb{R}^{n_2} \times \ldots \times \mathbb{R}^{n_t} \\
&: (v_1^T, v_2^T, \ldots, v_t^T)^T \mapsto (B_{X_1} v_1, B_{X_2} v_2, \ldots, B_{X_t} v_t),
\end{align*}
where $B_{X_i}^T$ and $B_{X_i}$ denote the mappings for manifold $\mathcal{M}_i$ at $X_i$, and $v_i \in \mathbb{R}^{n_i - d_i}$, $i = 1, \ldots, t$.
}

\section{Numerical Experiments} \label{sect:NumExp}


In this section, we use the sparse principle component analysis (SPCA) problem to test the proposed practical conditions on the  accuracy for solving the Riemannian proximal mapping~\eqref{e61}.
\subsection{Experimental Settings}

Since practically a sufficiently large $\tilde{L}$ is unknown, we dynamically increase its value by $\tilde{L} \gets 1.5 \tilde{L}$ if the search direction is not descent in the sense that back tracking algorithm $\alpha^{(i+1)} = 0.5 \alpha^{(i)}$ with $\alpha^{(0)} = 1$ for finding a step size fails for 5 iterations. In addition, the initial value of $\tilde{L}$ at $k+1$-th iteration, denoted by $\tilde{L}_{k+1}$, is given by the Barzilar-Borwein step size with safeguard:
\[
\tilde{L}_{k+1} = \min( \max( \left| \frac{\inner[]{y_k}{y_k}}{\inner[]{y_k}{s_k}} \right|, \tilde{L}_{\min} ), \tilde{L}_{\max}), 
\]
where $\tilde{L}_{\min} > 0$, $\tilde{L}_{\max} > 0$, $y_k = P_{\T_{x_{k}} \mathcal{M}} \grad f(x_{k+1}) - \grad f(x_k)$ and $s_k = \alpha \eta_{x_k}$. The value of $\tilde{L}_{0}$ is problem dependent and will be specified later.
The parameters are given by $\tilde{L}_{\min} = 10^{-3}$, $\tilde{L}_{\max} = \tilde{L}_0$, $\phi:\mathbb{R} \rightarrow \mathbb{R}: t \mapsto \sqrt{t}$, $\varepsilon_k = \frac{500}{(1+k)^{1.01}}$, and $\varrho = 100$.

Let IRPG-G, IRPG-U, and IRPG-L respectively denote Algorithm~\ref{a1} with the subproblem solved accurately enough in the sense that~\eqref{e58} and~\eqref{e60} hold, ~\eqref{e72} holds with $\psi(\varepsilon_k, \rho, \|\eta\|) = \varepsilon_k^2$, and~\eqref{e72} holds with $\psi(\varepsilon_k, \rho, \|\eta\|) = \min(\varepsilon_k^2, \varrho \|\eta\|^2)$. Unless otherwise indicated, IRPG-G stops if the value of $(\|\eta_{x_k}\| \tilde{L}_k)$ reduces at least by a factor of $10^{-3}$. IRPG-U and IRPG-L stop if their objective function values are smaller than the function value of the last iterate given by IRPG-G.


All the tested algorithms are implemented in the ROPTLIB package \cite{HAGH2016} using C++, with a MATLAB interface. The experiments are performed in Matlab R2018b on a 64 bit Ubuntu platform with 3.5GHz CPU (Intel Core i7-7800X).

\subsection{SPCA Test}~\label{sec:SPCA}

An optimization model for the sparse principle component analysis is given by
\begin{align} \label{spcast}
&\min_{X \in \mathrm{St}(p, n)} - \trace(X^T A^T A X) + \lambda \|X\|_1,
\end{align}
where $A \in \mathbb{R}^{m \times n}$ is the data matrix. This model is a penalized version of the ScoTLASS model introduced in~\cite{JoTrUd2003a} and it has been used in~\cite{CMSZ2019,HuaWei2019}. 

\paragraph{Basic settings} A matrix $\tilde{A} \in \mathbb{R}^{m \times n}$ is first generated such that its entries are drawn from the standard normal distribution. Then the matrix $A$ is created by shifting and normalized columns of $\tilde{A}$ such that the columns have mean zero and standard deviation one. The parameter $\tilde{L}_0$ is $2 \lambda_{\max}(A)^2$, where $\lambda_{\max}(A)$ denotes the largest singular value of $A$. The initial iterate is the leading $r$ right singular vectors of the matrix $A$. The Riemannian optimization tools including the Riemannian gradient, the retraction by polar decomposition, 
the inverse vector transport by differentiated the retraction, and the adjoint operator of the inverse vector transport by differentiated the retraction can be found in~\cite{HuaWei2019b}.

\paragraph{Empirical Observations} Figure~\ref{fig:SPCA} shows the performance of IRPG-G, IRPG-U, and IRPG-L with multiple values of $n$, $p$, and $\lambda$. Since IRPG-G, IRPG-U, and IRPG-L solve the Riemannian proximal mapping up to different accuracy, we find that IRPG-G takes notably more iterations than IRPG-U, and IRPG-U takes slightly more iterations than IRPG-L, which coincides with our theoretical results. Though IRPG-U and IRPG-L take fewer iterations, their computational times are still larger than that of IRPG-G due to the excessive cost on improving the accuracy of the Riemannian proximal mapping.



 
\begin{figure}[ht!]
\centerline{
\includegraphics[width=0.45\textwidth]{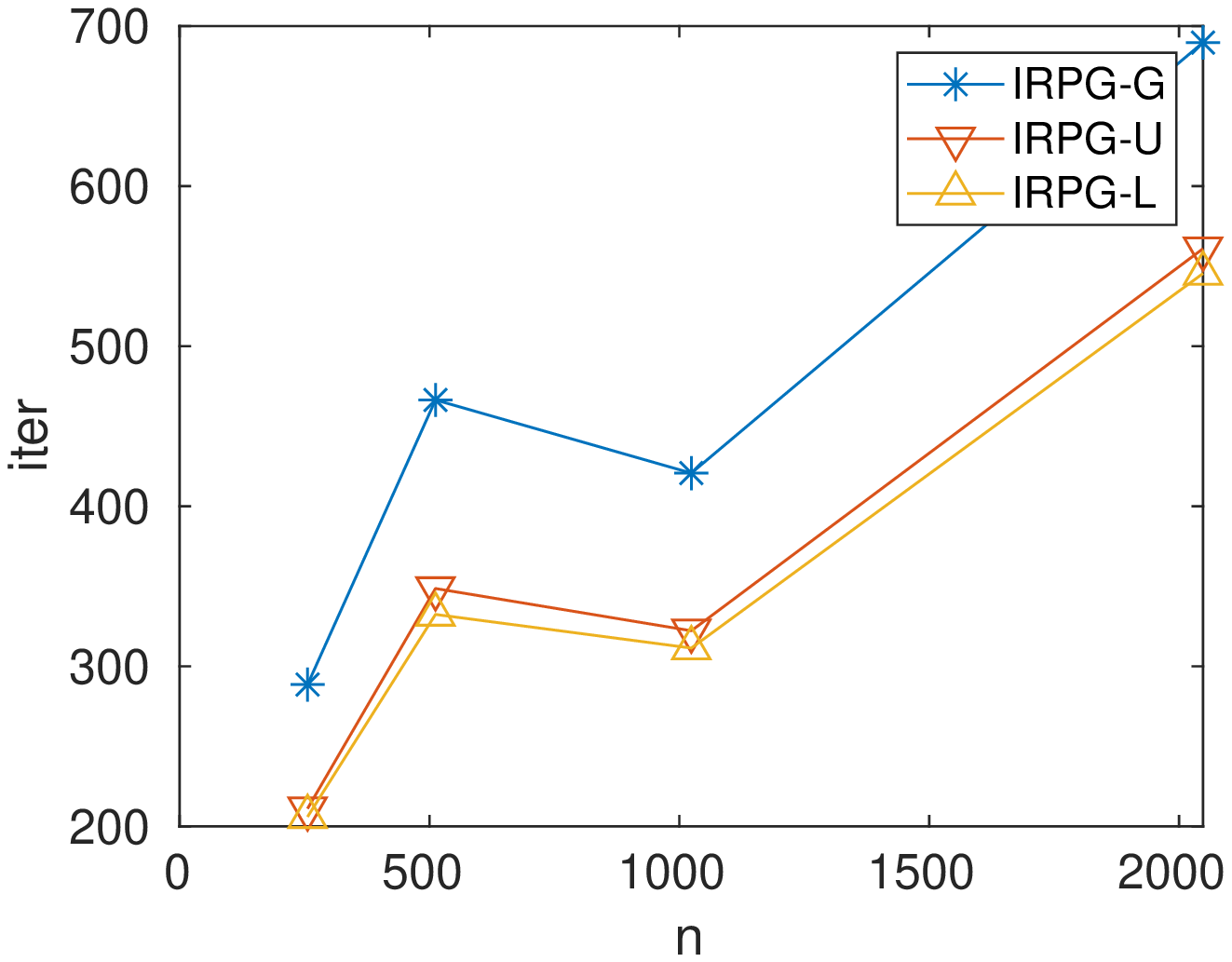}
\includegraphics[width=0.45\textwidth]{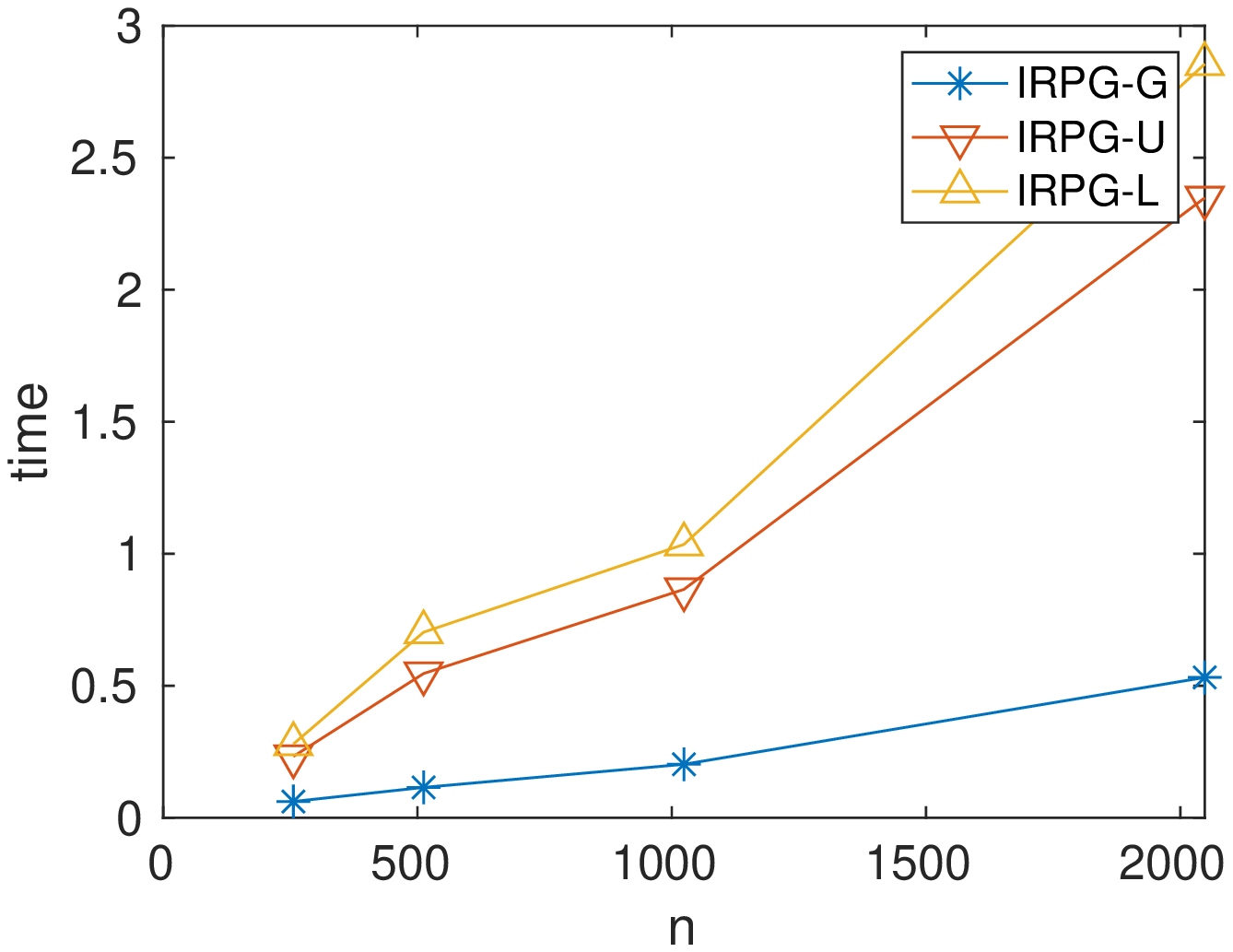}
}
\centerline{
\includegraphics[width=0.45\textwidth]{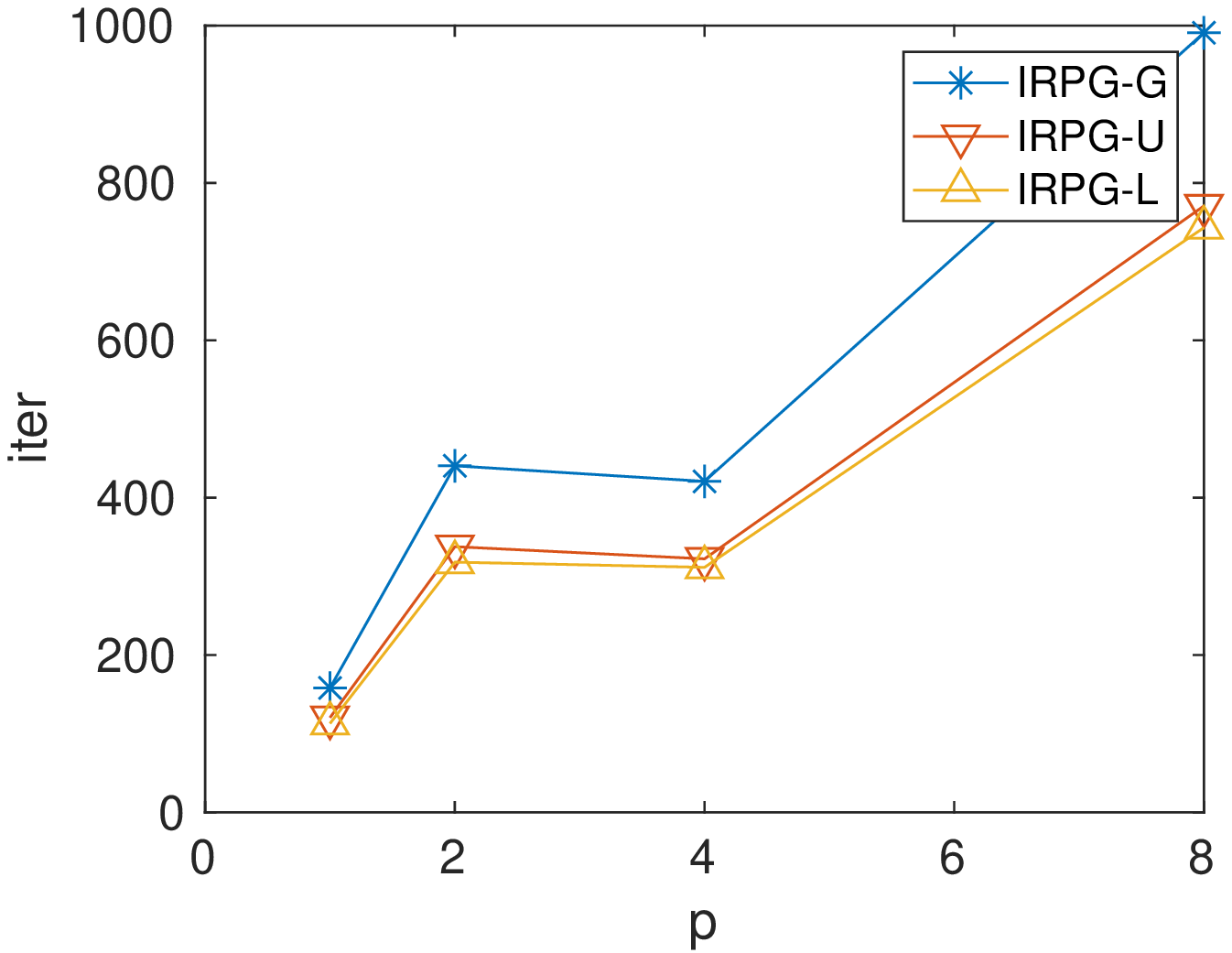}
\includegraphics[width=0.45\textwidth]{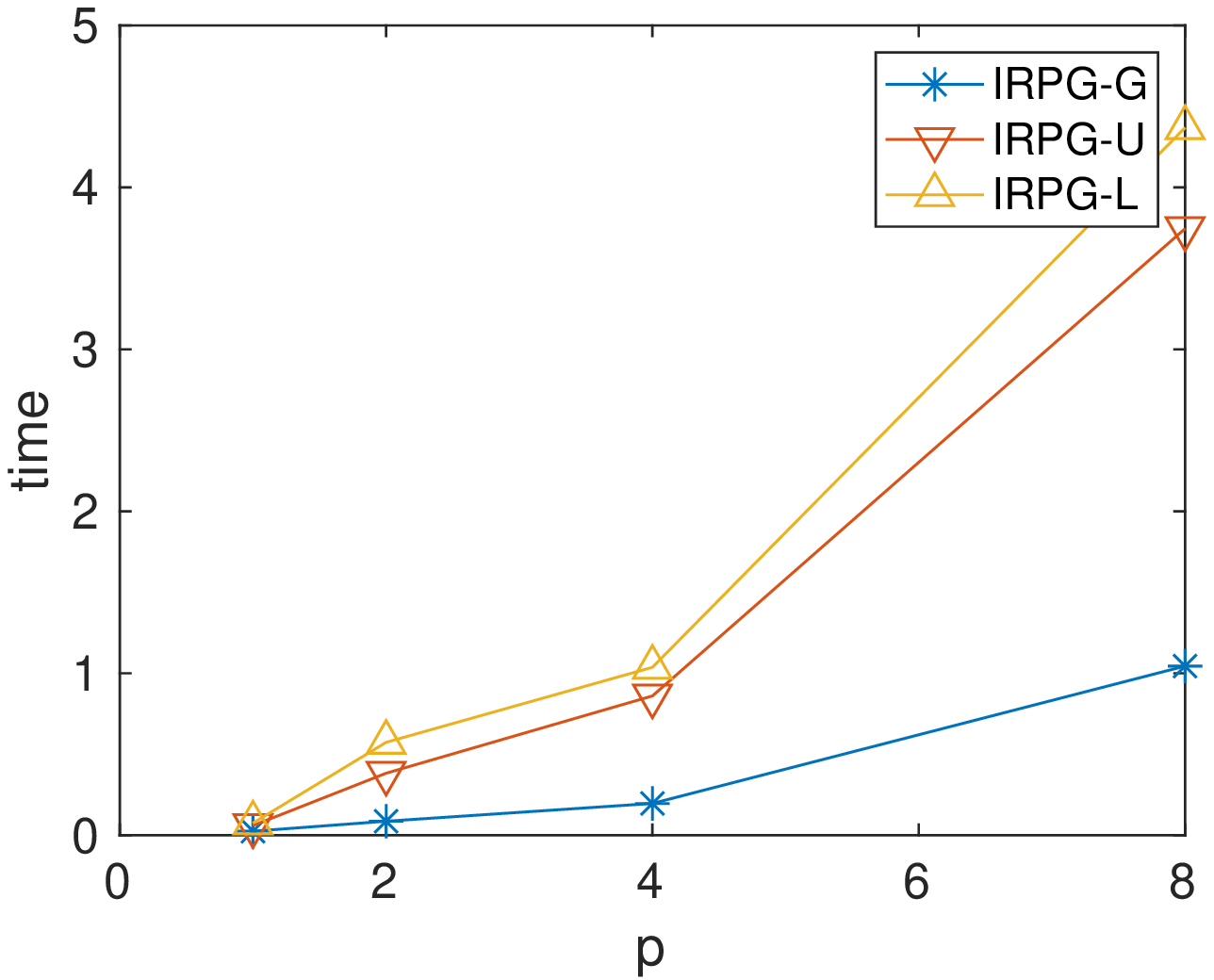}
}
\centerline{
\includegraphics[width=0.45\textwidth]{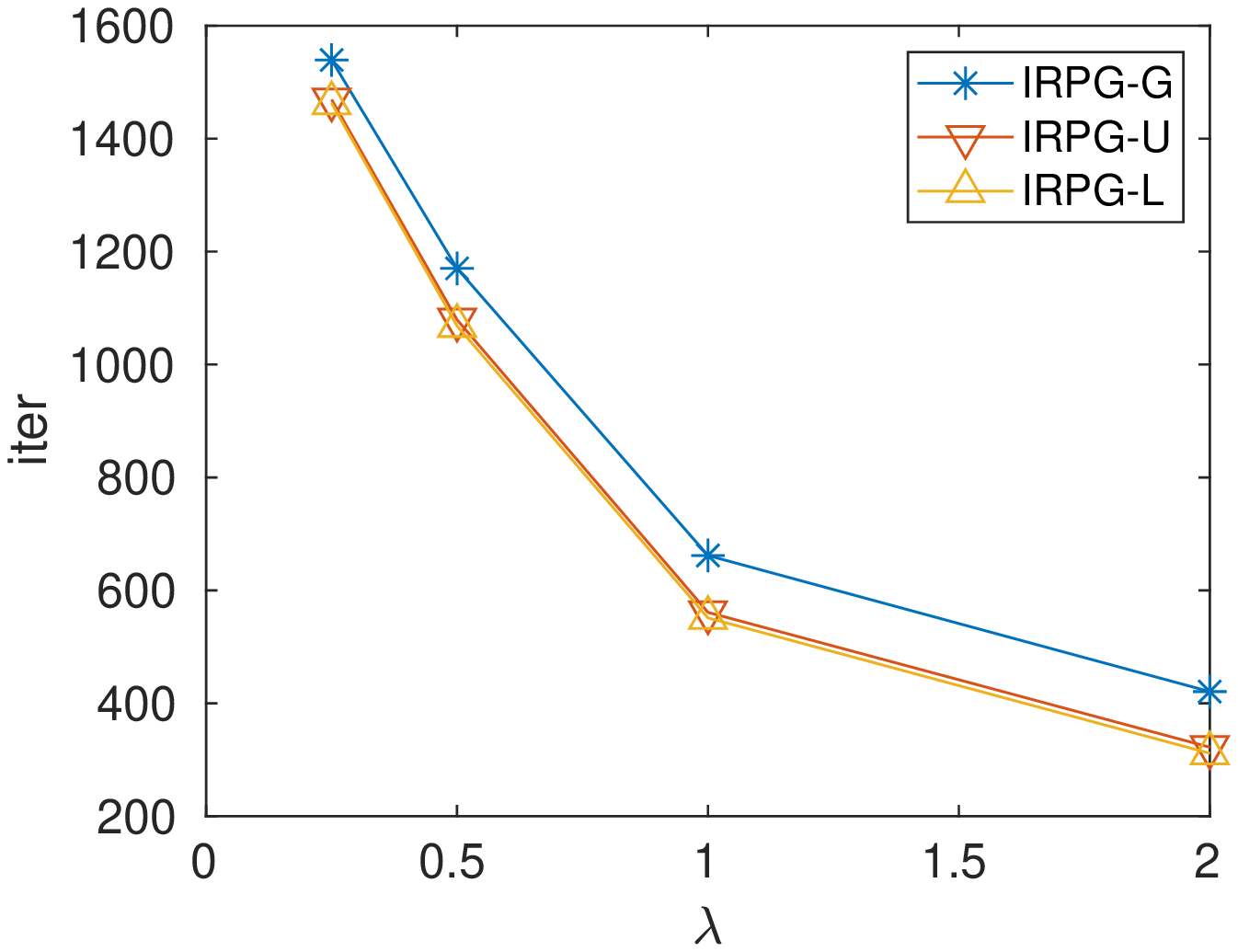}
\includegraphics[width=0.45\textwidth]{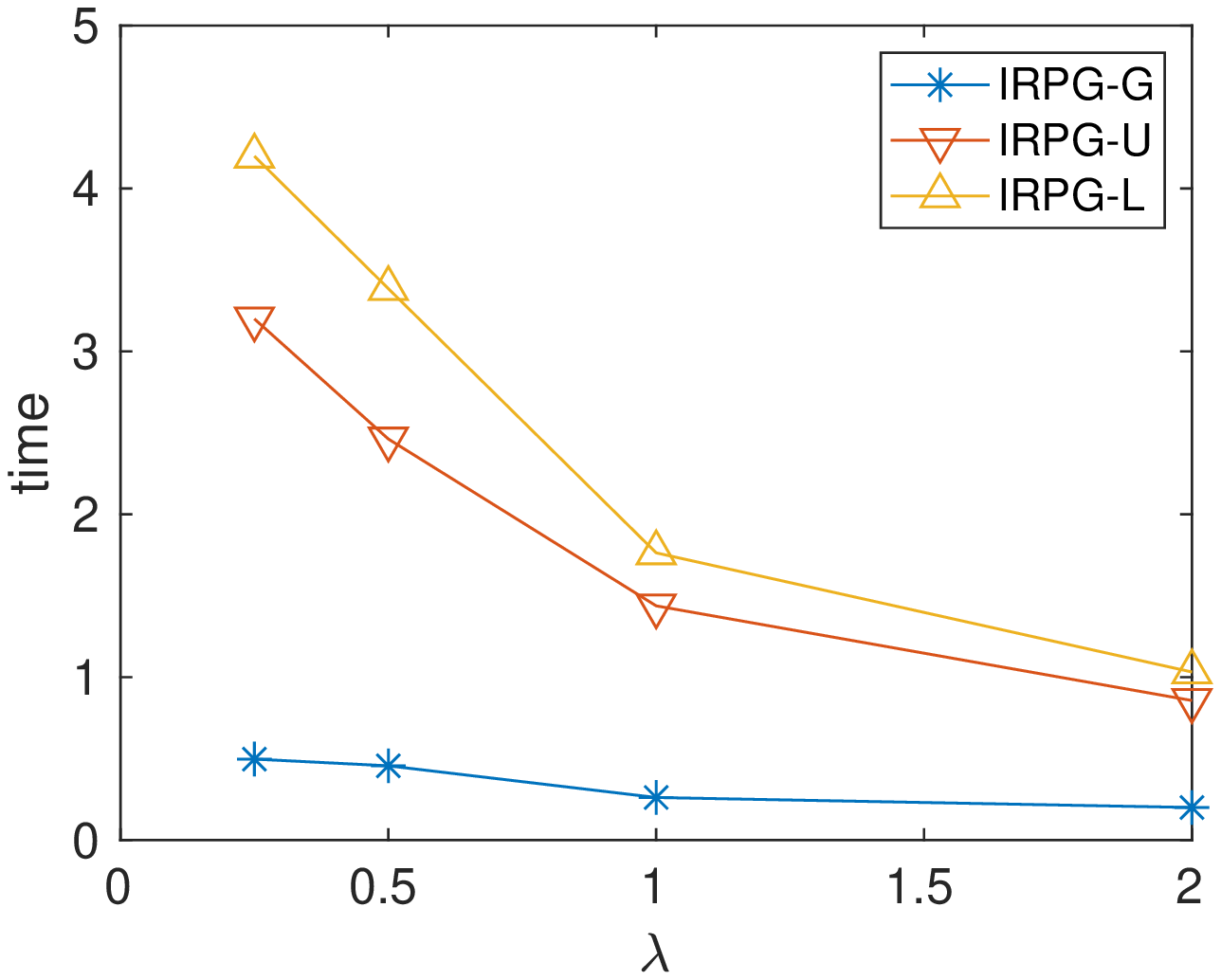}
}
\caption{
Average results of 10 random runs for SPCA. The same random seed is used when comparing the three algorithms. We choose the runs where the three algorithms find the same minimizer in the sense that the norm of the difference between the solutions is smaller than $10^{-2}$. ``time'' denotes the computational time in seconds. ``iter'' denotes the number of iterations. 
Top: multiple values $n = \{256, 512, 1024, 2048\}$ with $p = 4$, $m = 20$, and $\lambda = 2$; Middle: multiple values $p = \{1, 2, 4, 8\}$ with $n = 1024$, $m = 20$, and $\lambda = 2$; Bottom: Multiple values $\lambda = \{0.5, 1, 2, 4\}$ with $n = 1024$, $p = 4$, and $m = 20$.
}
\label{fig:SPCA}
\end{figure}

\section*{acknowledgements}

The authors would like to thank Liwei Zhang for discussions on perturbation analysis for optimization problems.

\bibliographystyle{alpha}
\bibliography{WHlibrary}

\end{document}